\documentclass[11pt]{report}
\usepackage{isuthesis,amsmath,amsthm,amssymb,graphs,latexsym,graphics,pb-diagram}
 \pdfoutput=1
\alternate
\theoremstyle{definition}
\newtheorem{thm}{Theorem}[section]
\newtheorem*{thm*}{Theorem}
\newtheorem{defn}[thm]{Definition}
\newtheorem{lemma}[thm]{Lemma}
\newtheorem{cor}[thm]{Corollary}
\allowdisplaybreaks[3]
\newcommand{\E}{\underline{E}_{n+1}^{\{1,3\}}}
\newcommand{\Ek}{\underline{E}_{k+1}^{\{1,3\}}}
\newcommand{\EE}{\underline{E}_{2^\omega}^{\{1,3\}}}
\newcommand{\F}{\mathcal{F}}
\newcommand{\sub}{\mathrm{Sb}(E)}
\newcommand{\Sub}{\underline{\mathrm{Sb}}(E)}
\newcommand{\RR}{\underline{\mathrm{Re}}}
\newcommand{\rr}{\mathrm{Re}}
\newcommand{\Id}{\mathrm{Id}}
\newcommand{\meet}{\wedge}

\newcommand{\conv}[1]{\breve{#1}}
\newcommand{\hh}{\textsf{\textbf{H}}}
\renewcommand{\ss}{\textsf{\textbf{S}}}
\newcommand{\sss}{\textsf{\textbf{S}}'}
\newcommand{\pp}{\textsf{\textbf{P}}}
\newcommand{\ii}{\textsf{\textbf{I}}}
\newcommand{\K}{\mathsf{K}}
\newcommand{\A}{\underline{A}}
\newcommand{\B}{\underline{B}}
\newcommand{\Q}{\underline{Q}}
\newcommand{\Fr}{\underline{\mathrm{Fr}}_\omega^\rho}
\newcommand{\Fra}{\underline{\mathrm{Fr}}_A^\rho}
\newcommand{\ds}{\displaystyle}
\newcommand{\ess}{\underline{S}} %%% \S is taken
\newcommand{\iso}{\cong}
\newcommand{\hg}{h^{-1}| \, g}
\newcommand{\1}{1\!\text{'}}
\newcommand{\0}{0\text{'}}
\newcommand{\C}{\mathcal{C}}
\newcommand{\Pt}{\text{Pt}_E}
\begin{document}
%auto-ignore
% Template Titlepage File
\title{On sets of first-order formulas axiomatizing representable relation algebras}
\author{Jeremy F. Alm}
\degree{MASTER OF SCIENCE} \major{Mathematics} \level{master's}
\mprof{Roger Maddux} 
% \members{Maria Axenovich \\ Paul Sacks \\
% Jonathan D. H. Smith \\ William Robinson}  
\members{Clifford Bergman \\ Howard Levine}
%\notice
% Add these additional lines for a Doctoral Dissertation
% \degree{DOCTOR OF PHILOSOPHY} \level{doctoral}
%format{dissertation}
%\committee{4}
%\members{Maria Axenovich \\ Paul Sacks \\ Jonathan D.
%H. Smith \\ William Robinson}
% Add these additional lines for a Creative Component
% - also comment out the \maketitle command
%\format{Creative Component}
%\submit{my graduate committee}
%\makenosigtitle
\date{2004}
\maketitle

\tableofcontents \addtocontents{toc}{\def\protect\@chapapp{}}
\cleardoublepage
%\listoftables
%\cleardoublepage
%\addcontentsline{toc}{chapter}{LIST OF FIGURES}
%\listoffigures
% Comment out the next line if NOT using chaptertitle
%\addtocontents{toc}{\def\protect\@chapapp{CHAPTER\ }}
%Optional thesis abstract
%\include{abstract}
\newpage
\pagenumbering{arabic}
%auto-ignore
% Introduction
\specialchapt{Preface}

\begin{center}
\footnotesize{\emph{The purpose of computing is insight, not numbers.} \\
\hspace{3in} --Richard Hamming \\

\bigskip

\emph{The pure mathematician knows that pure mathematics has an \\
end in itself which is more allied with philosophy.}\\
\hspace{1in} --Philip Jourdain, in the introduction to  Georg
Cantor's \\
\hspace{2in} \emph{Contributions to the Founding of the Theory of
Transfinite Numbers} }
\end{center}

The idea of a representation of an algebra is of great mathematical
interest, and indeed of philosophical interest, at least to
mathematicians. (Whether philosophers would exhibit an interest in
this more philosophical brand of mathematics I cannot say; whether
the public at large would take such an interest we can all say with
certainty.) Cayley's theorem is perhaps the most famous of the
representation theorems. It was, in a way, a great success: every
group was found actually to be isomorphic to a set of permutations
under the operation of functional composition.  So this first-order
formalization, the abstract algebraization of the notion of a set of
permutations, admits the isomorphic closure of the class of all such
sets of permutations.  That is, the sets of permutations could be
characterized formally, \emph{up to isomorphism}. And this is the
best that can be done with first-order languages, since they cannot
distinguish among isomorphic ``models."  It is only by our use of a
stronger language, our ``ordinary talk about mathematics", that we
distinguish among isomorphic non-identical models.  One wonders
whether there could (or should) be any formal logical system,
necessarily stronger than first-order predicate logic, that could
capture this distinction. Going further, De Morgan hoped for an
algebraic system that captured all the ``forms of thought." If this
seems almost na\"{i}ve, it is only because from our vantage point in
the history of mathematics there is much more to be seen than there
was in 1860. The whole development of mathematical logic and the
formalizations of set theory leave the author with a strong
impression of the indispensability of the full arsenal of our
mathematical language. It would seem that any ``model of thought"
designed for viewing the system from without, like first-order
languages, will always and necessarily be in some way inadequate.

Lest we be thought to despair: first-order logic provides for some
fascinating mathematics.  Indeed, it is the purpose of this thesis
to elucidate one small corner of that fascinating world.  We will
concern ourselves no longer with groups but with boolean algebras
with additional operations, those that generalize the composition
and conversion of binary relations.  Alfred Tarski laid out
algebraic axioms that hold in any field of binary relations, hoping
to find a first-order characterization of fields of binary relations
in the same way that sets of permutations were characterized by the
group axioms. It turned out that Tarski's axioms were insufficient,
as Roger Lyndon found a \emph{relation algebra} (an abstract algebra
satisfying Tarski's axioms) that was isomorphic to no field of
binary relations.  A natural question to ask at this point was, Did
Tarski ``forget" one?  Or two?  Could the definition of a relation
algebra be strengthened by adding an axiom or two (or several) so
that the models thereof would necessarily be isomorphic to fields of
binary relations?  Donald Monk answered this question negatively: no
finite axiom set would suffice. (Tarski had shown previously that
such an axiom set did indeed exist, but his proof did not speak to
the \emph{size} of such as set.)  So the algebras of ``real" binary
relations could be characterized, up to isomorphism, by algebraic
axioms, but only by infinitely many. This provides an interesting
contrast: our linguistic description of a field of binary relations
being fairly simple; a first-order characterization being
necessarily infinite, and being only a characterization up to
isomorphism, at that.  Thus the discrepancy between what we can
express using the full arsenal of our language and what we can
formalize in first-order logic is before us again.

It was stated above that the Cayley representation theorem was a
great success ``in a way."  There are two reasons for this qualified
statement.  Upon one of them we have already expounded. The other is
that the failure of a such a representation theorem to exist for
relation algebras, although it could be considered a disappointment,
admits so much interesting mathematics---mathematics that didn't
need to be developed for groups.  This ``failure" has inspired fifty
years of interesting mathematical research, some of which is
contained in these pages. Had Tarski's axioms ``succeeded" in the
way that the group axioms did, this thesis would not have been
written.  Of course, we must not go so far as to be \emph{glad} for
this ``failure"; we must take the world as it is, and not let our
romantic conceptions of what ``would be more interesting or less
interesting" skew our view of it.  But we can relish what we get,
and expect more and more interesting mathematics, for as long as
people are inclined to its study.

%auto-ignore
% Chapter 0 of the Thesis Template File
\chapter{Introduction}
\section{General algebra}
Let $I$ be a non-empty set, and let $\rho: I \longrightarrow\omega$
be a function. An \emph{algebra} $\A$ is a non-empty set $A$
together with functions $f_i : A^{\rho(i)} \longrightarrow A$, $i\in
I$. $\rho:I \longrightarrow \omega$ is the \emph{type} of $\A$.
Algebras are \emph{similar} if they have the same type.  Rings with
identity and fields are similar, for example, but rings and groups
are not (the operations do not ``match up"). So for a group
$\underline{G}$ we write $\underline{G}=\langle G, *, {}^{-1}, e
\rangle$; for a ring $\underline{R}$, write $\underline{R}=\langle
R, +, -, \cdot, 0, 1 \rangle$.  $A\subseteq B$ will denote ordinary
set inclusion, while $\A\subseteq\B$ will denote the subalgebra
relation.  For an algebra $\A$, and $X\subseteq A$, let
$\text{Sg}^{\A} (X)$ denote the subalgebra in $\A$ generated by $X$.
For a homomorphism between similar algebras we will always write
$h:\A \longrightarrow \B$ (as opposed to $h:A\longrightarrow B$).
While $h:A\longrightarrow B$ could mean any function between the
underlying sets of the algebras, $h:\A\longrightarrow \B$ will
\emph{always} denote a homomorphism. $|$ will stand for composition
of binary relations. (See Def. 2.1.1) For example,
$\A\iso|\subseteq\B$ means that $\A$ is isomorphic to a subalgebra
of $\B$, so in particular there is some $\underline{C}\subseteq\B$
with $\underline{C}\iso\A$.

We will often work with classes of similar algebras.  $\mathsf{V}$
will denote the class of all sets.  $\mathsf{Id}$ will denote the
universal identity relation, $\{\langle x,x \rangle :
x\in\mathsf{V}\}$.

For a class $\K$ of similar algebras, let
\begin{align*}
\ii\K &=\{\A : \A \text{ is isomorphic to some member of }\K\} \\
 \hh\K &=\{\A : \A \text{ is a homomorphic image of
some member of
} \K\} \\
\ss\K &= \{\A : \A \text{ is isomorphic to a subalgebra of some
member of }\K\} \\
\ss'\K &=\{\A : \A \text{ is a subalgebra of some member of } \K\}
\\
\pp\K &= \{\A : \A \text{ is isomorphic to a product of members of
}\K\}
\end{align*}

For a class $\K$, the following identities and inclusions hold:
\begin{align*}
\ii\ii\K &= \ii\K \\
\hh\hh\K &= \hh\K \\
\ss\ss\K &= \ss\K \\
\ss'\ss'\K &= \ss'\K \\
\ii\ss'\K &= \ss\K \\
\pp\pp\K &= \pp\K \\
\ss\hh\K &\subseteq \hh\ss\K \\
\pp\ss\K &\subseteq \ss\pp\K \\
\pp\hh\K &\subseteq \hh\pp\K \\
\end{align*}

The last three lines are not equalities since equality can fail to
hold for various classes $\K$.  Applying the operators $\hh,\ss,\pp$
to a class $\K$, $\hh\ss\pp\K$ yields the largest class.  For more
on class operators see chapter 0 of \cite{CA}.

An \emph{equational class} is a class $\K = \{ \A : \A \models
\Sigma \}$, where $\Sigma$ is a set of equations and $\models$
denotes satisfaction of formulas.  Birkhoff's theorem says that $\K$
is an equational class iff $\K=\hh\K=\ss\K=\pp\K$.  See the appendix
for more on Birkhoff's theorem.

We have a general correspondence between homomorphisms
$h:\A\longrightarrow \B$ and congruences $\mathcal{C}\subseteq
A\times A$.  Given $h:\A\longrightarrow \B$, let $\mathcal{C} =
h|h^{-1}\subseteq A\times A$, where $|$ denotes composition of
binary relations.  (See definition 2.1.1)  Then $\mathcal{C}$ is a
congruence relation. Conversely, given a congruence relation
$\mathcal{C}\subseteq A\times A$, let $h:\A\longrightarrow
\A/\mathcal{C}$ given by $a\in A \longmapsto a/\mathcal{C}$, where
$\A/\mathcal{C} = \{a/\mathcal{C} :  a \in A \}$ and $a/\mathcal{C}$
is the equivalence class of $a$.

For more on general algebra see \cite{BS}.

\section{Boolean algebra}

Boolean algebras are generalizations of algebras of sets under the
operations of union (generalized by +), intersection (generalized by
$\cdot$), complementation (generalized by $-$), and the constants
$\emptyset$ and $U$, the universe (generalized by 0 and 1).
 $\textsf{BA}$ is the class of all boolean algebras.  We write $\B
 = \langle B,+,\cdot, - , 0,1 \rangle$ for $\B\in\textsf{BA}$.
 $\B$ satisfies all of the following:

\pagebreak[3]
\begin{align*}
x+(y+z) &= (x+y)+z   &&\text{(assoc. of +)}\\
x\cdot (y\cdot z) &= (x\cdot y)\cdot z &&\text{(assoc. of $\cdot$)}\\
x+y &= y+x  &&\text{(comm. of +)}\\
x\cdot y &= y\cdot x &&\text{(comm. of $\cdot$)}\\
x\cdot (y+z) &= x\cdot y + x\cdot z &&\text{(dist. of
$\cdot$)}\\
x+(y\cdot z) &= (x+y)\cdot(x+z) &&\text{(dist. of +)}\\
x+(x\cdot y) &= x &&\text{(abs. of +)}\\
x\cdot(x+y) &= x &&\text{(abs. of $\cdot$)}\\
x+ \, -x &= 1  &&\text{(comp. of +)}\\
x\cdot \, -x &= 0  &&\text{(comp. of $\cdot$)}
\end{align*}
The order of operations is $-$, $\cdot$, $+$.   Each of these
axioms is easily seen to hold for algebras of sets, where $+$ is
interpreted as union, etc. Sometimes $+$ is called \emph{join},
and $\cdot$ \emph{meet}.

The two-element boolean algebra is called \emph{trivial}.  The
one-element boolean algebra is called \emph{degenerate}.  Every
non-degenerate boolean algebra has a homomorphism onto the trivial
algebra.

There is an alternate (but equivalent) definition, in which boolean
algebras are given fewer fundamental operations.  (See \cite{Madd}.)
Let $\B = \langle B,+,- \rangle$, and $\B$ satisfies

\begin{align*}
x+(y+z) &= (x+y)+z   &\text{(assoc. of +)}\\
x+y &= y+x  &\text{(comm. of +)}\\
x &= \overline{\bar{x} + \bar{y}} + \overline{\bar{x} + y}
&\text{(Huntington's axiom)}
\end{align*}
Now we can define $x\cdot y = \overline{\bar{x} + \bar{y}}$.  The
constants 0 and 1 can be defined by  $1 = x+ -x$ and $0 = x\cdot -x$
for any $x$.  (One can prove that $x+-x=y+-y$ for any $x,y$, so the
definitions of 0 and 1 makes sense.)  This first definition of
\textsf{BA}s given is a little more natural, while the second is
simpler.  We adopt the second. Here we have written $\bar{x}$
instead of $-x$. We will do this when it is convenient.

We define a partial order $\leq$ where $x\leq y$ iff $x+y=y$.
(Think of $X\subseteq Y$ iff $ X\cup Y = Y$ for sets.)
 $+, \cdot$ are monotone, i.e. $x\leq y \Rightarrow x+z\leq y+z$.

We can also define the \emph{symmetric difference}, $x\vartriangle y
= x\cdot\bar{y} + \bar{x}\cdot y$. One can prove $ x\vartriangle y =
0 \iff x=y$.

An \emph{atom} is a minimal non-zero element in the partial
ordering.  An algebra is \emph{atomic} if every non-zero element has
an atom below it.  An algebra is said to be \emph{complete} if
arbitrary meets and joins exist, i.e. given an arbitrary $X\subseteq
B$, $\inf X $  and  $\sup X$ (with respect to the partial ordering)
exist. We write $\Sigma X = \sup X$ and $\prod X = \inf X$.  In a
 atomic boolean algebra, we can write every element as the
join of the atoms below it:
\[
x=\sum_{\substack{a\in\text{At }\B
\\ a \leq x}} a
\]
where $\text{At }\B$ is the set of all atoms of $\B$.

A homomorphism $h:\B \longrightarrow \B'$ is a map that preserves
$+$ and $-$.  A \emph{(boolean-algebraic) ideal} in an algebra $\B$
is a set $I\subseteq B$ such that $x,y \in I \Rightarrow x+y \in I$
and $y\in I, \ x\leq y \Rightarrow x\in I$.\footnote{A
\emph{set-theoretic ideal} is a set $I\subseteq \mathcal{P}(U)$ such
that $X,Y \in I \Rightarrow X\cup Y \in I$ and $Y\in I, \ X\subseteq
Y \Rightarrow X\in I$, where $\mathcal{P}(U)$ is the power set of
$U$.  Boolean-algebraic ideals are the abstract analogue of
set-theoretic ideals.  Despite the difference in cosmetics, they are
really ``the same."}  For a homomorphism $h$, let ker$\, h$ denote
the pre-image of 0 under $h$.  Then every kernel is an ideal, and
every ideal is the kernel of some homomorphism.  Given an ideal $I$,
$\mathcal{C} = \{\langle x,y \rangle : x\vartriangle y \in I \}$ is
a congruence relation on $\B$. Similarly, given a congruence
$\mathcal{C}$, the equivalence class of zero is an ideal.

For more on boolean algebra, see \cite{BA} and chapter 2 of
\cite{HH}.

\section{Ultraproducts}

A \emph{(set-theoretic) filter} $\mathcal{F}$ on a set $U$ is
$\mathcal{F}\subseteq \mathcal{P}(U)$ such that $X\in\mathcal{F},
X\subseteq Y \Rightarrow Y\in\mathcal{F}$ and $X,Y\in\mathcal{F}
\Rightarrow X\cap Y \in\mathcal{F}$.  An \emph{ultrafilter} is a
maximal proper filter.  $\mathcal{F}$ is an ultrafilter iff for
all $X\in\mathcal{P}(U)$, $X\in\mathcal{F}$ or $U\setminus X
\in\mathcal{F}$ but not both.

Let $\mathcal{F}$ be an ultrafilter on an indexing set $I$.  Let
$\prod_{i\in I} A_i$ be a product of sets, and $f,g \in \prod_{i\in
I} A_i$.  We define an equivalence relation $\sim_\mathcal{F}$ by
$f/\mathcal{F} = g/\mathcal{F}$) iff $\{i \in I : f(i)= g(i) \} \in
\mathcal{F}$. We will write the equivalence class of $f$ as $f/
\mathcal{F}$ rather than $f/ \sim_\mathcal{F}$.   Then

\[
\prod_{i\in I} A_i \Big / \mathcal{F} = \{f/\mathcal{F} : f \in
\prod A_i \}
\]

is called an \emph{ultraproduct}. Considering an ultrafilter to be
the collection of ``big" subsets of $I$, then $f\sim_\mathcal{F} g$
iff $f$ and $g$ agree ``almost everywhere," or on a ``big set."

If the sets $A_i$ are algebras instead of just sets (and then we
write $\A_i$), then we define the operations in the ultraproduct in
the usual way on representatives from the equivalence classes:
$(f/\mathcal{F})
* (g/\mathcal{F}) := (f*g)/\mathcal{F}$ for a binary operation
$*$ and where $(f*g)$ denotes the ``pointwise'' product.  Here is an
important result, known as Los' Lemma:

Let $\varphi$ be a first-order sentence.\footnote{For a definition
see \cite{CK}.  An example would be the group-theoretic sentence
$(\forall x )(\exists y)( x\cdot y = 1)$} Then
\[
\prod_{i\in I} \A_i \Big / \mathcal{F} \models \varphi \iff \{
i\in I : \A_i \models\varphi \} \in \mathcal{F}
\]
In particular, ultraproducts preserve satisfaction of sentences.

For more on the ultraproduct construction, see \cite{CK}.

%auto-ignore
% Chapter 2 of the Thesis Template File
\chapter{Relation algebras}

\renewcommand{\theenumi}{\roman{enumi}}

\section{Definition of RA}

In order to motivate the definition of a relation algebra, we
discuss those structures of which they are the abstract analogue.

\subsection{Algebras of relations}

Let Sb($X$) denote the power set of $X$, and let Re($X$) denote
the power set of $X\times X$.

\begin{defn} An \emph{algebra of binary relations} (or \emph{proper relation algebra}) is an algebra $\langle A,
\cup,\, \bar{\ }, \, |, \, {}^{-1}, \, \Id_E \rangle $, where $A
\subseteq \sub $ for a nonempty equivalence relation $E$, and
$\Id_E=\mathsf{Id}\cap E$. The operation $|$ , called \emph{relative
multiplication} or \emph{composition}, is given by $ R|S = \{\langle
x,z \rangle : \exists y \ \langle x,y \rangle \in R \meet \langle
y,z \rangle \in S \}$, and the operation ${}^{-1}$ is given by
$R^{-1} = \{ \langle x,y \rangle :
\langle y,x \rangle \in R \}$.\\

The set of all such algebras is denoted by $\textsf{PRA}$ for proper
relation algebra, and $\textsf{PRA} = \ss' \{\Sub : E= E|E^{-1} \}$.
Here and everywhere, $\Sub$ denotes the algebra
over the set $\sub$.\\
 \end{defn}
\begin{thm} A set  $E\subseteq U\times U$ (for some $U$) is an equivalence relation iff
$E=E|E^{-1}$.\end{thm}

\begin{proof}
By definition, $E$ is an equivalence relation precisely when $E$
satisfies $E^{-1}=E$ (symmetry) and $E|E\subseteq E$ (transitivity).
So to prove the ``only if" direction, let $E$ be an equivalence
relation.  Then $E=E|E^{-1}\subseteq E$.  Also, if $xEy$ then
$xExEy$, since $E$ is reflexive over its field, and so we have
$E\subseteq E|E=E|E^{-1}$.  Thus $E=E|E^{-1}$ as desired.  For the
``if" direction, let $E=E|E^{-1}$. Then
$E^{-1}=(E|E^{-1})^{-1}=E|E^{-1}=E$, so $E$ is symmetric. Also,
$E|E=E|E^{-1}=E$, so $E$ is transitive.
\end{proof}

 Algebras of
the form $\RR (U)=\langle \text{Re}(U), \cup,\, \bar{\ }, \, |, \,
{}^{-1}, \, \Id_{U\times U} \rangle $ and their subalgebras are
called \emph{square} \textsf{PRA}s, since the boolean unit has the
form $U\times U$.

\subsection{Examples}

The following Hasse diagram shows the boolean structure of a
particular proper subalgebra of $\RR(\{x,y\})$.  This algebra is
\emph{square} (the boolean unit is of the form $U\times U$) but
not \emph{full} (it is not a power set algebra).

\bigskip

\begin{graph}(10,10)(-7,-5)
\roundnode{1}(0,5) \roundnode{div}(-3,0) \roundnode{id}(3,0)
\roundnode{0}(0,-5) \edge{1}{id} \edge{1}{div} \edge{0}{id}
\edge{0}{div} \nodetext{1}{$\{\langle x,x \rangle, \langle x,y
\rangle, \langle y,x \rangle, \langle y,y \rangle \}$}
\nodetext{id}{$\{\langle x,x \rangle,  \langle y,y \rangle \}$}
\nodetext{div}{$\{\langle x,y \rangle, \langle y,x \rangle \}$}
\nodetext{0}{$\varnothing$}
\end{graph}

\bigskip

\bigskip

Another example is $\Sub$, where $E=\{\langle x,x \rangle, \langle
y,y \rangle, \langle z,z \rangle \}$.  In this instance, the boolean
unit is also the relational identity.  This algebra is \emph{full}
but not \emph{square}.  See the following Hasse diagram:

\bigskip

%\drawgraphsfalse
\begin{graph}(12,12)(-7,-6)
\roundnode{1}(0,6) \roundnode{2}(-3,2) \roundnode{3}(0,2)
\roundnode{4}(3,2)  \roundnode{8}(-3,-2) \roundnode{9}(0,-2)
\roundnode{10}(3,-2) \roundnode{0}(0,-6) \edge{1}{2} \edge{1}{3}
\edge{1}{4} \edge{2}{8} \edge{2}{9} \edge{3}{8} \edge{3}{10}
\edge{4}{9} \edge{4}{10} \edge{0}{8} \edge{0}{9} \edge{0}{10}
\nodetext{1}{$\{\langle x,x \rangle, \langle y,y \rangle, \langle
z,z \rangle \}$} \nodetext{2}{$\{\langle x,x \rangle, \langle y,y
\rangle \}$} \nodetext{3}{$\{\langle x,x \rangle, \langle z,z
\rangle \}$} \nodetext{4}{$\{\langle y,y \rangle, \langle z,z
\rangle \}$} \nodetext{8}{$\{\langle x,x \rangle\}$}
\nodetext{9}{$\{\langle y,y \rangle \}$} \nodetext{10}{$\{\langle
z,z \rangle \}$} \nodetext{0}{$\varnothing$}
\end{graph}

\bigskip

 Any $\RR (U)$ can be written $\Sub$, where
$E=U\times U$.  It is an interesting theorem that any subalgebra
of an $\Sub$ is a subalgebra of a direct product of algebras of
the form $\RR (U)$.  Thus we have the following decomposition
theorem.

\begin{thm} $\ss\{\Sub : E=E|E^{-1} \} = \ss\pp\{\RR (U) : U\in \mathsf{V} \}$. \end{thm}

\begin{proof}  We simply show $\ii\{\Sub : E=E|E^{-1} \} = \pp\{\RR (U) : U\in \mathsf{V} \}$.  Let $E=\bigcup_{\alpha\in J} (U_\alpha\times U_\alpha)$ where the
$U_\alpha \,$s are the disjoint equivalence classes of $E$.  Then
the maps

\begin{align*}
R\in \sub &\longmapsto \langle R\cap (U_\alpha \times U_\alpha) :
\alpha \in J
\rangle \\
\langle R_\alpha : \alpha \in J \rangle \in \prod_{\alpha\in J}
\rr (U_\alpha) &\longmapsto
\bigcup_{\alpha\in J} R_\alpha, \ \ \text{where the $R_\alpha \,$s are disjoint}\\
\end{align*}

establish the desired correspondence.
\end{proof}

Thus every proper relation algebra has a decomposition into a
subalgebra of a product of square algebras.

\subsection{Abstract relation algebras}

\begin{defn} A \emph{relation algebra} is an algebra $\A = \langle
A, +, -, \, ;, \, \conv{}, \, \1 \rangle$ that satisfies

\begin{align}
(x+y)+z &= x+(y+z) \\
x+y &= y+x \\
x &= \overline{\bar{x}+y} + \overline{\bar{x}+\bar{y}}\\
x;(y;z) &= (x;y);z \\
(x+y);z &= x;z + y;z \\
x;\1 &= x \\
\conv{\conv{x}} &= x \\
(x+y)\conv{} &= \conv{x} + \conv{y} \\
(x;y)\conv{} &= \conv{y};\conv{x} \\
\bar{y} + \conv{x};\overline{x;y} &= \bar{y}
\end{align}

The first three axioms say that $\A$ is a boolean algebra with
additional operations.  The remaining axioms are relational
identities that hold in every \textsf{PRA}, translated into the
abstract algebraic language. The class of all relation algebras is
denoted by \textsf{RA}.
\end{defn}

We can define a partial order $\leq$ on an \textsf{RA} by $x\leq y$
iff $x+y = y$ (iff $x\cdot y=x$).  Also, $x<y$ will mean $x\leq y$
and $x\neq y$.  An element $a$ is called an \emph{atom} if $a> 0$
and $(\forall x)(x< a \Rightarrow x=0)$.  An \textsf{RA} is called
\emph{atomic} if every nonzero element has an atom below it.  An
\textsf{RA} is called \emph{symmetric} if conversion is the identity
function $(x=\conv{x})$.  An \textsf{RA} is called \emph{integral}
if $ x=0$ or $y=0$ whenever $x;y=0$.  This last condition is
equivalent to the condition that $\1$ be an atom.  (See Th. 2.2.9)

\subsection{An example}

Consider the following \textsf{RA} with boolean structure as
follows:

\bigskip

\begin{graph}(12,12)(-7,-6)
\roundnode{1}(0,6) \roundnode{2}(-3,2) \roundnode{3}(0,2)
\roundnode{4}(3,2)  \roundnode{8}(-3,-2) \roundnode{9}(0,-2)
\roundnode{10}(3,-2) \roundnode{0}(0,-6) \edge{1}{2} \edge{1}{3}
\edge{1}{4} \edge{2}{8} \edge{2}{9} \edge{3}{8} \edge{3}{10}
\edge{4}{9} \edge{4}{10} \edge{0}{8} \edge{0}{9} \edge{0}{10}
\nodetext{1}{1} \nodetext{2}{$a+b$} \nodetext{3}{$a+c$}
\nodetext{4}{$b+c$} \nodetext{8}{$a$} \nodetext{9}{$b$}
\nodetext{10}{$c$} \nodetext{0}{0}
\end{graph}

\bigskip

All elements of this algebra are self-converse $(\conv{x}=x)$.
Relative multiplication is given by $x;y=x\cdot y$.  If such an
equational definition of ; is not available, composition for
finite \textsf{RA}s can be specified by a multiplication table on
the atoms.  For this example we have

\[
\begin{tabular}{r|lll}
; &   $a$ & $b$ & $c$ \\
\hline
 $a$  & $a$ & 0 & $0$ \\
 $b$  & 0 & $b$ & 0 \\
 $c$  & 0 & 0 & $c$ \\
\end{tabular}
\]

It is sufficient to specify $;$ on the atoms because in a finite
relation algebra every non-zero element is the join of all of the
(finitely many) atoms below it, and $;$ distributes over $+$.  (Of
course, each entry in the table should be the join of some atoms.)
For example, if $x=a+b$ and $y=b+c$, then to compute $x;y$ we do the
following:  $x;y=(a+b);(b+c)= a;b + a;c + b;b + b;c$.  Now each of
these terms in the sum can be read off the multiplication table for
atoms.

It is easy to see that this abstract relation algebra is isomorphic
to the second algebra given in 2.1.2, which is a \textsf{PRA}.

%%%%%%%%%%%%%%%%%%%%%%%%%%%%%%%%%%%%%%%%%%%%%%%%%%%%%%%%%%%%%%%%%%%%%%%%%%%%%%%%%%%%
\section{Arithmetic in RA}
We derive some useful results that hold in \textsf{RA}.  First we
show that left-distributivity of ; over $+$ follows from the
axioms.  Note that only right-distributivity is assumed
explicitly.

\begin{thm} $x;(y+z) = x;y + x;z$. \end{thm}

We use $(a;b)\,\conv{} = \conv{b};\conv{a}$ to ``turn around" the
composition, and use right-distributivity.

\begin{proof}
\begin{align*}
x;(y+z) &= \big([x;(y+z)]\,\conv{}\, \big)\, \conv{} \qquad \text{by (1.7)} \\
        &= \big((y+z)\,\conv{};\conv{x}\big)\conv{} \qquad \text{by (1.9)} \\
        &= \big((\conv{y}+\conv{z});\conv{x}\big)\,\conv{} \qquad \text{by (1.8)} \\
        &= \big(\conv{y};\conv{x} + \conv{z};\conv{x}\big)\conv{} \qquad \text{by (1.5)} \\
        &= \big((x;y)\,\conv{} + (x;z)\,\conv{}\,\big)\conv{} \qquad\text{by (1.9)} \\
        &= \big([x;y + x;z]\,\conv{}\,\big)\,\conv{} \qquad\text{by (1.8)} \\
        &= x;y + x;z \qquad \text{by (1.7)}
\end{align*}\qedhere

\end{proof}

\begin{thm}
  The operation $\conv{}$ is an automorphism of the boolean reduct
  of any $\A\in\textsf{RA}$. In particular,
  $\conv{\bar{x}}=\bar{\conv{x}}$ (or
  $-\conv{x}=(\bar{x})\,\conv{}\,)$.
\end{thm}

\begin{proof}
The operation $\conv{}$ is bijective, since $\conv{\conv{x}} = x$.
$(x+y)\,\conv{} = \conv{x} + \conv{y}$ is axiom (1.8).
\[
x\leq y \iff x+y = y \iff \conv{x} + \conv{y} = (x+y)\,\conv{} =
\conv{y}  \iff \conv{x} \leq \conv{y}
\]
so conversion is order-preserving (or ``monotone").

Now since $\conv{}$ preserves order, we have $1\conv{}\leq 1$ and $1
= (1\conv{}\,)\,\conv{} \leq 1\conv{}$, so $1\conv{}=1$.  Also, we
have $0\,\conv{}\geq 0$ and $0 = (0\,\conv{}\,)\,\conv{} \geq
0\,\conv{}$, so $0\,\conv{} = 0$.  Hence $\conv{}$ preserves 0 and
1.

To prove $\conv{\bar{x}}=\bar{\conv{x}}$, first note that $x +
\bar{x} =1$, so $\conv{x} + \conv{\bar{x}} = (x+\bar{x})\,\conv{} =
1\conv{}=1$.  Therefore $\bar{\conv{x}} \leq \conv{\bar{x}}$.
Similarly, $\conv{x} + \bar{\conv{x}} = 1$, so by a parallel
argument $\bar{x} \leq(\bar{\conv{x}})\,\conv{}$.  By monotonicity
of conversion, we have $\conv{\bar{x}} \leq(\bar{\conv{x}})\,\conv{}
\ \conv{} = \bar{\conv{x}}$. Therefore
$\conv{\bar{x}}=\bar{\conv{x}}$.

Hence conversion is an automorphism of the boolean reduct.
\qedhere
\end{proof}

\begin{thm}
  If $x\leq y$ then  $x;z \leq y;z$ and
  $z;x \leq z;y$ \qquad (monotonicity of ; )
\end{thm}

\begin{proof}
  Assume $x\leq y$.  Then $x;z \leq x;z + y;z = (x+y);z =
  y;z$.  Also, $z;x \leq z;x + z;y = z;(x+y)=z;y$.  \qedhere
\end{proof}

The following theorem is known as the \emph{Peircean Law}, after C.
S. Peirce (pronounced ``purse").

\begin{thm}
  $x;y\cdot\conv{z} = 0 \iff
y;z\cdot\conv{x} = 0$ \label{PL}
\end{thm}

\begin{proof}
By (1.10), we have $\conv{a};\overline{a;b}\leq\bar{b}$ for any $a$
and $b$.

Suppose that $x;y\cdot \conv{z} = 0$.  This is equivalent to
$\overline{x;y}\geq \conv{z}$.  Then

$(\overline{x;y})\,\conv{}\geq z$ by Th 2.2.2, $ -(x;y)\,\conv{}
\geq z$ by Th 2.2.2 again, and $-(\conv{y};\conv{x})\geq z$ by
(1.9).  Finally we have $y;\overline{\conv{y};\conv{x}}\geq y;z$ by
monotonicity of ;.

%Hence $-(\conv{y};\conv{x})\geq z$.   Then let $a=\conv{y}$,
%$b=\conv{x}$ in (1.10), and we have
%$y;\overline{\conv{y};\conv{x}}\leq \bar{\conv{x}}$ by 1.7,
%equivalently $y;\overline{\conv{y};\conv{x}}\cdot \conv{x} = 0$.
%Then by monotonicity of ; we have $\overline{\conv{y};\conv{x}}\geq
%z \Rightarrow y;\overline{\conv{y};\conv{x}}\geq y;z$.

(1.10) tells us that $y;\overline{\conv{y};\conv{x}}\leq -\conv{x}$.
Combining the last two inequalities, we get $y;z \leq
y;\overline{\conv{y};\conv{x}}\leq -\conv{x}$, and hence $y;z\cdot
\conv{x}=0$.

 The converse follows by an alternate
assignment of the roles of $x,y,z$.  \qedhere
\end{proof}

\begin{thm} $ \1\,\conv{} = \1, \ \ \1;x=x, \ \ 0;x=0=x;0$

  Note that neither of the first two were assumed explicitly. ($\1$ was
  assumed to be a \emph{right} identity.)
\end{thm}

\begin{proof}
$\1\,\conv{}=\1\,\conv{};\1= \1\,\conv{};(\1\,\conv{}\,)\,\conv{} =
(\1\,\conv{};\1)\,\conv{}  = \1\,\conv{}\ \conv{} = \1$.  Also,
$\1;x = (\conv{x};\1\,\conv{}\,)\,\conv{} = (\conv{x};\1)\,\conv{} =
\conv{\conv{x}}=x$. \\

To prove $0;x=0$, note that by the previous theorem we have
$y;x\cdot \conv{z}=0 \iff x;z\cdot \conv{y}=0$.  Now let $x$ be
arbitrary, let $y=0=\conv{0}$, and let $z=1=\conv{1}$.  Then we get
$0;x\cdot 1 = 0 \iff x;1 \cdot 0 = 0$.  Now the right side of this
biconditional is always true; therefore the left is also.  Thus $0;x
= 0;x \cdot 1 = 0$, as desired.  The derivation of $x;0=0$ is
similar.

\qedhere
\end{proof}

\begin{thm} A symmetric algebra is commutative.  \end{thm}

\begin{proof}
  Suppose $\conv{x}=x$ for all $x$.  Then $x;y =
  (x;y)\conv{}\ \conv{} =
  (x;y)\,\conv{} \stackrel{(1.9)}{=}  \conv{y};\conv{x} =  y;x$. \qedhere
\end{proof}

\begin{thm}
  $ x\leq x;\conv{x};x$
\end{thm}

\begin{proof}
%$x;\conv{x}=0 \Longrightarrow x;\conv{x}\cdot \1 = 0
%\stackrel{(\ref{PL})}{\Longrightarrow} x=\1;x\cdot x = 0$.  So
%$x;\conv{x} \neq 0$ for $x\neq 0$.

%Suppose we have $a\neq 0$.  Let $0 < x \leq a$.  Then
%$x;\conv{x}\neq 0$.  By monotonicity, $x;\conv{x} \leq a;\conv{x}$
%and $x;\conv{x}\leq a;\conv{a}$.  Then
%$(a;\conv{x})\cdot(a;\conv{a})\geq x;\conv{x}\neq 0$.  By
%(\ref{PL}), $a;\conv{a};a\cdot x \neq 0$ $(\star)$.  Suppose $a\nleq
%a;\conv{a};a.$  Let $x=a \cdot\overline{a;\conv{a};a} \neq 0$.  But
%then $x\cdot a;\conv{a};a=0$, which denies $(\star)$.  So it must be
%that $a\leq a;\conv{a};a$. \qedhere

\begin{align*}
  x\cdot y;z &= x \cdot y;(z\cdot 1)\\
             &= x \cdot y;(z\cdot (\conv{y};x
             +\overline{\conv{y};x}))\\
             &= x \cdot y;(z\cdot \conv{y};x)
             + x \cdot y;(\overline{\conv{y};x}\cdot z))\\
             &= x \cdot y;(z\cdot \conv{y};x) + 0
\end{align*}
To justify the last step, $x \cdot y;(\overline{\conv{y};x}\cdot
z))=0$, note that by (1.10) we have $y;\overline{\conv{y};x}\leq
\bar{x}$ and hence $x \cdot y;(\overline{\conv{y};x}\cdot z))\leq
x\cdot\bar{x}=0$.

Now we use the above, reassigning the roles of $y$ and $z$: let
$y=x$, $z=\1$.  Then
\begin{align*}
  x &=x\cdot x;\1\\
    &=x\cdot x;(\1\cdot \conv{x};x) &&\text{by the above}\\
    &=x\cdot x;\conv{x};x\\
    &\leq x;\conv{x};x
\end{align*}
\end{proof}

\begin{thm}
  If $x,y\leq \1$, then $\conv{x}=x$ and $x;y=x\cdot y$.
\end{thm}

\begin{proof}
By the previous theorem and monotonicity, $x \leq x;\conv{x};x \leq
\1;\conv{x};\1 = \conv{x}$.  So $x\leq\conv{x}$.  Then
$\conv{x}\leq\conv{\conv{x}}=x$, and $x=\conv{x}$.

For the second result, $x;y\leq \1;y = y$ and $x;y\leq x;\1=x$ by
monotonicity, so $x;y\leq x\cdot y$.  Also, by monotonicity and
the previous theorem $x\cdot y \leq (x\cdot y);(x\cdot
y)\,\conv{};(x\cdot y) \leq x;(x\cdot y)\,\conv{};y \leq
x;(\1;\1)\,\conv{};y = x;y$.  Hence $x;y=x\cdot y$. \qedhere
\end{proof}

\begin{thm}
  Let $\A\in\textsf{RA}$ be non-degenerate.  Then $\A$ is integral
  iff $\1\in \text{At}\, \A$.
\end{thm}

\begin{proof}
First, show $\1 \notin \text{At } \A \Longrightarrow \A$ not
integral.

By hypothesis, $\exists x \ 0<x<\1$.  Let $y=\bar{x}\cdot \1$.
Note that $y\neq 0$.  Then $x;y=x\cdot y = x\cdot \bar{x} \cdot \1
= 0$.  So $\A$ is not integral.

Conversely, suppose $\1 \in \text{At } \A$.  We want $x;y\neq 0$
for $x\neq 0 \neq y$.

First we have $0\neq \conv{x}=\1;\conv{x}\cdot\conv{x}$.  By
(\ref{PL}), $\conv{x};x\cdot \1 \neq 0$.  Since $\1$ is an atom, we
have $\conv{x};x\geq \1$.  Therefore by monotonicity we have
$\1;y\leq (\conv{x};x);y\leq (\conv{x};1);1=\conv{x};1$.  So $y =
\1;y\leq \conv{x};1$, and $\conv{x};1\cdot y \neq 0$.  Then by
(\ref{PL}), $x;y\cdot 1 \neq 0$, and $x;y\neq 0$. \qedhere
\end{proof}

\begin{thm}
  $(1;x;1)\,\conv{}=1;x;1$
\end{thm}

\begin{proof}

   $x \leq x;\conv{x};x \leq 1;\conv{x};1 $ by 2.2.7 and
   monotonicity of composition.  So $x\leq 1;\conv{x};1$.  Again by monotonicity , $1;x;1\leq
1;1;\conv{x};1;1=1;\conv{x};1$.  Hence $1;x;1\leq 1;\conv{x};1$
$(\spadesuit)$. By letting $\conv{x}$ take the role of $x$ in
$(\spadesuit)$, we get $1;\conv{x};1\leq 1;\conv{\conv{x}};1=1;x;1$.
Hence $1;\conv{x};1 =1;x;1$, and $(1;x;1)\,\conv{}=1;\conv{x};1
=1;x;1$ as desired.
\end{proof}

It is useful to refer to $1;x;1$ as the \emph{closure} of $x$,
especially in a proper relation algebra, where
$E|R|E=\bigcup\{U_\alpha \times U_\alpha : U_\alpha\text{ is an
equivalence class of }E, \  (U_\alpha \times U_\alpha)\cap R
\neq\emptyset\}$.

\begin{defn}
  Let $\A\in\textsf{RA}$, and $I\subseteq A$.  $I$ is said to be a
  \emph{relational ideal} if
  \begin{enumerate}
\item $y\in I, \ x \leq y \ \Rightarrow x\in I$

\item $x,y\in I \Rightarrow x+y\in I$

\item $x\in I \Rightarrow 1;x, \, x;1, \, \conv{x} \in I$

\end{enumerate}
\end{defn}

Note that an ideal $I$ is proper iff $1\notin I$, since $I$ is
``closed going down" (see i. above).  It is a straightforward
exercise to show that iii. above is equivalent to $x\in I
\Rightarrow 1;x;1\in I$.

\begin{defn}
  An algebra $\A$ is \emph{simple} if the only homomorphisms from $\A$ onto similar algebras are either isomorphisms or else
  are mappings from $\A$ to the degenerate (1-element) algebra.
\end{defn}

This next theorem provides a useful characterization of the simple
relation algebras.

\begin{thm}
Let $\A\in\textsf{RA}$.  Then $\A$ is simple iff for all $x\neq 0$,
$1;x;1=1$.
\end{thm}

Note that if $\A$ is simple and $g$ is a homomorphism with domain
$\A$, then $\text{ker } g = \{x\in A: g(x)=g(0)\}$ is either $\{0\}$
or $A$. Thus the only relational ideals on $\A$ are $\{0\}$ and $A$.

\begin{proof}
We prove both directions by contrapositive.  Suppose $\A$ is not
simple.  Then there is a relational ideal $I$ on $A$ such that
$\{0\}\subsetneq I\subsetneq A$.  Thus there is some $x\in I$,
$x\neq 0$ and $1;x;1\in I$.  But $I$ is proper, so $1\notin I$, and
$1;x;1<1$.

Conversely, suppose that there is some $x\neq 0$ so that $1;x;1<1$.
Then $I=\{z : z\leq 1;x;1 \}$ is a relational ideal, and
$\{0\}\subsetneq I\subsetneq A$.
\end{proof}

%%%%%%%%%%%%%%%%%%%%%%%%%%%%%%%%%%%%%%%%%%%%%%%%%%%%%%%%%%%%%%%%%%%%%%%%%%%%%%%%%%%%
\section{Representable relation algebras}

\textsf{RA}s are algebraic generalizations of \textsf{PRA}s.  It
is natural to ask whether every \textsf{RA} is isomorphic to some
\textsf{PRA}.

\begin{defn}  A relation algebra is said to be
\emph{representable} if it is isomorphic to some proper relation
algebra.  The class of all representable relation algebras is
denoted by \textsf{RRA}. \end{defn}

So we have $\textsf{RRA}= \ii\textsf{PRA} = \ii\sss\{\Sub:
E=E|E^{-1}\} = \ss \{\Sub:
E=E|E^{-1}\}$.\\

\begin{thm}[Lyndon, 1950]  \textsf{RRA}$\neq$\textsf{RA}. \end{thm}

\begin{proof}  We exhibit a non-representable relation algebra.
Lyndon found a large non-representable relation algebra.  The
following algebra, which is the smallest, is due to MacKenzie.

Let $\A$ be an algebra with four atoms $\1, a, \conv{a}, b$
($b=\conv{b}$).  The multiplication table for diversity atoms is as
follows:

\bigskip

\begin{tabular}{r|lll}
; &   $a$ & $\conv{a}$ & $b$ \\
\hline
 $a$  & $a$ & 1 & $a+b$ \\
 $\conv{a}$  & 1 & $\conv{a}$ & $\conv{a}+b$ \\
 $b$  & $a+b$ & $\conv{a}+b$ & $\bar{b}$ \\
\end{tabular}

\bigskip

We will show that this cannot be the multiplication table for a
proper relation algebra.  Suppose that $\1, a, \conv{a}, b$ are
real relations, and that $\1$ is an identity relation.  All these
atoms are non-zero, so they all contain a pair.  Let $\langle x,y
\rangle \in b$.

\begin{graph}(6,4)(-6,-2)
\grapharrowlength{.2} \roundnode{x}(-2,0) \roundnode{y}(0,0)
\edge{x}{y} \autonodetext{x}[w]{$x$} \autonodetext{y}[e]{$y$}
\bowtext{x}{y}{0.12}{$b$}
\end{graph}

$b\leq a;\conv{a}=\conv{a};a$, so there exist $w,v$ so that

\bigskip

\bigskip

\begin{graph}(6,4)(-6,-2)
\grapharrowlength{.2} \roundnode{x}(-2,0) \roundnode{y}(0,0)
\edge{x}{y} \autonodetext{x}[w]{$x$} \autonodetext{y}[se]{$y$}
\bowtext{x}{y}{0.1}{$b$} \roundnode{w}(0,2) \roundnode{v}(0,-2)
\diredge{w}{x} \diredge{w}{y} \diredge{x}{v} \diredge{y}{v}
\autonodetext{w}[n]{$w$} \autonodetext{v}[s]{$v$}
\bowtext{w}{x}{-0.1}{$a$} \bowtext{w}{y}{0.1}{$a$}
\bowtext{x}{v}{-0.1}{$a$} \bowtext{y}{v}{0.1}{$a$}
\end{graph}

\bigskip

Then $\langle w,v \rangle \in a;a=a$, so we can add the edge

\bigskip

\begin{graph}(6,4)(-6,-1.5)
\grapharrowlength{.2} \roundnode{x}(-2,0) \roundnode{y}(0,0)
\edge{x}{y} \autonodetext{x}[w]{$x$} \autonodetext{y}[se]{$y$}
\bowtext{x}{y}{0.1}{$b$} \roundnode{w}(0,2) \roundnode{v}(0,-2)
\diredge{w}{x} \diredge{w}{y} \diredge{x}{v} \diredge{y}{v}
\autonodetext{w}[n]{$w$} \autonodetext{v}[s]{$v$}
\bowtext{w}{x}{-0.1}{$a$} \bowtext{w}{y}{0.1}{$a$}
\bowtext{x}{v}{-0.1}{$a$} \bowtext{y}{v}{0.1}{$a$}
\dirbow{w}{v}{0.2} \bowtext{w}{v}{0.25}{$a$}
\end{graph}

\bigskip

\bigskip

Now $a\leq b;b$, so there exists $z$ which is distinct from $v,w,x$
such that

\bigskip

\bigskip

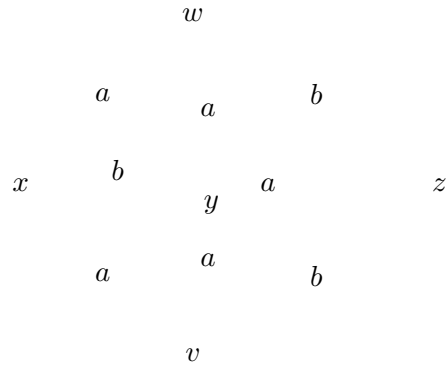
\begin{graph}(6,4)(-6,-2)
\grapharrowlength{.2} \roundnode{x}(-2,0) \roundnode{y}(0,0)
\edge{x}{y} \autonodetext{x}[w]{$x$} \autonodetext{y}[se]{$y$}
\bowtext{x}{y}{0.1}{$b$} \roundnode{w}(0,2) \roundnode{v}(0,-2)
\diredge{w}{x} \diredge{w}{y} \diredge{x}{v} \diredge{y}{v}
\autonodetext{w}[n]{$w$} \autonodetext{v}[s]{$v$}
\bowtext{w}{x}{-0.1}{$a$} \bowtext{w}{y}{0.1}{$a$}
\bowtext{x}{v}{-0.1}{$a$} \bowtext{y}{v}{0.1}{$a$}
\dirbow{w}{v}{0.2} \bowtext{w}{v}{0.25}{$a$} \roundnode{z}(3,0)
\edge{w}{z} \edge{v}{z} \bowtext{w}{z}{0.07}{$b$}
\bowtext{z}{v}{0.07}{$b$} \autonodetext{z}[e]{$z$}
\end{graph}

\bigskip

\bigskip

Note that $z\neq v,w,x$ since $\langle z,x\rangle \in b;a\leq \0$
and $\langle z,v\rangle, \langle z,w\rangle  \in b\leq \0$.  Now
since  $\langle z,x \rangle \in b;a$ and $\A$ is a finite algebra,
 there is an atom that contains $\langle z,x \rangle$;\footnote{If $\A$ were not finite this might not be the
case!} likewise for $\langle z,y \rangle$. Hence we have the
following edges that need labels:

\bigskip

\bigskip

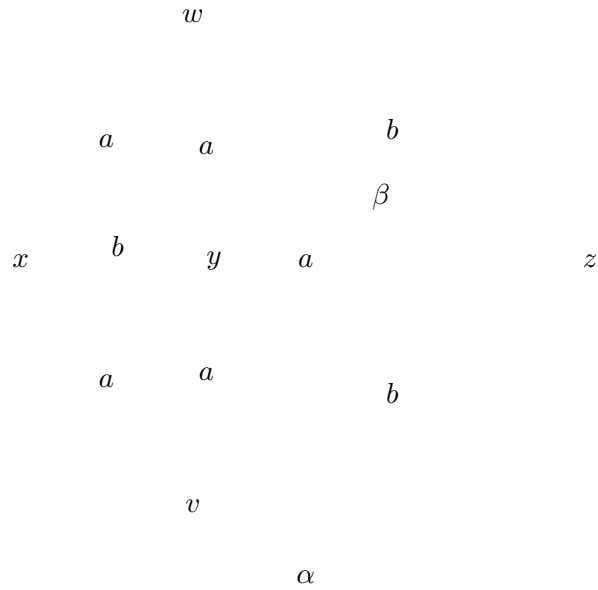
\begin{graph}(7,7)(-6,-4)
\grapharrowlength{.2} \roundnode{w}(0,3) \roundnode{x}(-2,0)
\roundnode{y}(0,0) \roundnode{v}(0,-3) \roundnode{z}(5,0)
\edge{w}{z} \edge{v}{z} \edge{x}{y}  \diredge{w}{x} \edge{w}{z}
\diredge{w}{y} \diredge{x}{v} \diredge{y}{v} \dirbow{w}{v}{0.2}
\autonodetext{w}[n]{$w$} \autonodetext{x}[w]{$x$}
\autonodetext{y}[e]{$y$} \autonodetext{z}[e]{$z$}
\autonodetext{v}[s]{$v$} \bowtext{w}{x}{-0.05}{$a$}
\bowtext{w}{y}{0.06}{$a$} \bowtext{x}{v}{-0.05}{$a$}
\bowtext{y}{v}{0.06}{$a$} \bowtext{w}{v}{0.25}{$a$}
\bowtext{w}{z}{0.05}{$b$} \bowtext{z}{v}{0.05}{$b$}
\bowtext{x}{y}{0.1}{$b$} \bow{z}{x}{.55}[\graphlinedash{3 1}]
\bow{z}{y}{-.12}[\graphlinedash{3 1}] \bowtext{z}{x}{.6}{$\alpha$}
\bowtext{z}{y}{-.17}{$\beta$}
\end{graph}

\bigskip

\bigskip

%Now we isolate part of the picture:

%\bigskip

%\bigskip

%\begin{graph}(8,4)(-7,-2)
%\grapharrowlength{.2} \roundnode{x}(-2,2) \roundnode{y}(2,2)
%\roundnode{w}(-4,0) \roundnode{z}(0,0) \roundnode{w2}(4,0)
%\autonodetext{x}[nw]{$x$} \autonodetext{y}[ne]{$y$}
%\autonodetext{w}[sw]{$w$} \autonodetext{z}[s]{$z$}
%\autonodetext{w2}[se]{$w$} \edge{x}{y} \edge{w}{z} \edge{z}{w2}
%\diredge{w}{x} \diredge{w2}{y} \edge{z}{x}[\graphlinedash{3 1}]
%\edge{z}{y}[\graphlinedash{3 1}] \bowtext{w}{x}{0.07}{$a$}
%\bowtext{w2}{y}{-0.07}{$a$} \bowtext{x}{y}{0.08}{$b$}
%\bowtext{w}{z}{-0.07}{$b$} \bowtext{z}{w2}{-0.07}{$b$}
%\bowtext{z}{x}{0.08}{$\alpha$} \bowtext{z}{y}{-0.1}{$\beta$}
%\freetext(5,1){(1)}
%\end{graph}

%\bigskip

%\begin{graph}(8,4)(-7,-2)
%\grapharrowlength{.2} \roundnode{x}(-2,2) \roundnode{y}(2,2)
%\roundnode{v}(-4,0) \roundnode{z}(0,0) \roundnode{v2}(4,0)
%\autonodetext{x}[nw]{$x$} \autonodetext{y}[ne]{$y$}
%\autonodetext{v}[sw]{$v$} \autonodetext{z}[s]{$z$}
%\autonodetext{v2}[se]{$v$} \edge{x}{y} \edge{v}{z} \edge{z}{v2}
%\diredge{x}{v} \diredge{y}{v2} \edge{z}{x}[\graphlinedash{3 1}]
%\edge{z}{y}[\graphlinedash{3 1}] \bowtext{v}{x}{0.07}{$a$}
%\bowtext{v2}{y}{-0.07}{$a$} \bowtext{x}{y}{0.08}{$b$}
%\bowtext{v}{z}{-0.07}{$b$} \bowtext{z}{v2}{-0.07}{$b$}
%\bowtext{z}{x}{0.08}{$\alpha$} \bowtext{z}{y}{-0.1}{$\beta$}
%\freetext(5,1){(2)}
%\end{graph}

%\bigskip

%\bigskip

Now $\langle x,z\rangle, \, \langle y,z\rangle \in \conv{a};b\cdot
a;b=b$. Therefore the edges marked $\alpha, \beta$ can be labeled
$b$.  But then $x, y,$ and $z$ form a ``monochromatic triangle":

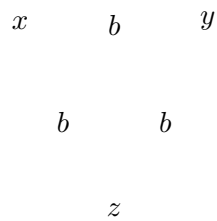
\bigskip

\begin{graph}(4,4)(-7,-2)
\grapharrowlength{.2} \roundnode{x}(-1,1) \roundnode{y}(1,1)
\roundnode{z}(0,-1) \edge{x}{y} \edge{x}{z} \edge{y}{z}
\bowtext{x}{y}{0.1}{$b$} \bowtext{y}{z}{0.09}{$b$}
\bowtext{z}{x}{0.09}{$b$} \autonodetext{x}[nw]{$x$}
\autonodetext{y}[ne]{$y$} \autonodetext{z}[s]{$z$}
\end{graph}

But then $b\cdot b;b \neq 0$, which is incompatible with the
multiplication table, which says that $b;b=\bar{b}$.  Hence $\A$ is
not representable.
\end{proof}

%auto-ignore
% Chapter 2 of the Thesis Template File
%   which includes bibliographic references.
\chapter{RRA is an equational class}

In this chapter we will show that \textsf{RRA} is closed under
$\hh$, $\ss$, and $\pp$.  Thus by Birkhoff's theorem the set of
equations true in \textsf{RRA} axiomatizes the class.

This theorem has an interesting history.  In \cite{Lyn50} Roger
Lyndon published a proof that \textsf{RRA} was not axiomatizable by
quantifier-free formulas---namely equations,
quasi-equations\footnote{A \emph{quasi-equation} is a
quantifier-free formula of the form $\varepsilon_0 \meet
\varepsilon_1 \meet \ldots \meet \varepsilon_{n-1} \Longrightarrow
\varepsilon_n$, where $\varepsilon_0, \ldots, \varepsilon_n$ are all
equations.}, and the like.  Five years later Tarski published a
result that showed that \textsf{RRA} was axiomatizable by equations
in \cite{Tar55}!  It turned out that Lyndon had made a mistake and
that one of the algebras that he used in his proof which was thought
not to be representable was in fact representable.

\section{Closure under subalgebras and products}

\begin{thm}
$\textsf{RRA}=\ss\textsf{RRA}$.
\end{thm}

\begin{proof}
$\ss\textsf{RRA}=\ss\ss\{\Sub : E=E|E^{-1} \}=\ss\{\Sub :
E=E|E^{-1} \}=\textsf{RRA}$.
\end{proof}

\begin{thm}
$\textsf{RRA}=\pp\textsf{RRA}$.
\end{thm}

\begin{proof}
Recall Th 2.1.2, which says that $\ss\{\Sub : E=E|E^{-1} \} =
\ss\pp\{\RR (U)\}$.

Then $\pp\textsf{RRA}=\pp\ss\pp\{\RR(U)\}
\stackrel{(\star)}{\subseteq}
\ss\pp\pp\{\RR(U)\}=\ss\pp\{\RR(U)\}=\textsf{RRA}$, where ($\star$)
holds since $\pp\ss\subseteq\ss\pp$ in general.
\end{proof}

\section{Closure under homomorphic images}

Closure under $\hh$ turns out to be as challenging as closure under
$\ss$ and $\pp$ was easy.  The proof that we give here is due to
Roger Maddux, and was presented by him in a universal algebra course
at Iowa State. This proof also appears in \cite{Madd}.  We have a
series of lemmas, theorems, and definitions. Our first goal is to
demonstrate that $\hh\ss\{\Sub : E=E|E^{-1} \}=\ss\hh\{\Sub :
E=E|E^{-1} \}$.

\begin{defn} We say that an algebra $\A'$ is \emph{congruence
extensile}, or that it has the \emph{congruence extension property},
if for all $\A\subseteq\A'$ and all congruences $\C\subseteq A\times
A$, $\C$ extends to a congruence $\C'\subseteq A'\times A'$ so that
$\C = \C' \cap (A\times A)$.
\end{defn}

\begin{lemma}
Suppose $\A'$ has the congruence extension property.  Then
$\hh\ss\{\A'\} = \ss\hh\{\A'\}$.
\end{lemma}

\begin{proof}
We know in general that $\ss\hh\subseteq\hh\ss$.  So suppose that
$\B \in \hh\ss\{\A'\}$.  So there is some $\A \subseteq \A'$, and
$h:\A \twoheadrightarrow \B$.\footnote{Recall that $h:\A
\longrightarrow \B$ (as opposed to $h:A\longrightarrow B$) will
always denote a \emph{homomorphism}.}  Let $\C=h|h^{-1}$. Extend
$\C$ to $\C'\subseteq A'\times A'$. $\C'$ induces a homomorphism
$h':\A'\twoheadrightarrow \A' / \C'$. Then $h'\cap (A\times
\textsf{V}):\A\twoheadrightarrow\A / \C$ is a homomorphism, since
$\C=\C'\cap(A\times A)$.  So $\A / \C \in \ss\hh\{\A'\}$.  But
$\B\iso \A / \C$, so $\B \in \ss\hh\{\A'\}$.
\end{proof}

\begin{lemma}
Let $\A'\in\textsf{RA}$.  Then $\A'$ has the congruence extension
property.
\end{lemma}

\begin{proof}
Let $\A\subseteq \A'$.  Let $\C\subseteq A\times A$ be a congruence.
Let $I$ be the equivalence class of the boolean zero.  $I$ is a
relational ideal, i.e. satisfies i.--iii. below:

\renewcommand{\theenumi}{\roman{enumi}}

\begin{enumerate}
\item $y\in I, \ x \leq y \ \Rightarrow x\in I$

\item $x,y\in I \Rightarrow x+y\in I$

\item $x\in I \Rightarrow 1;x, \, x;1, \, \conv{x} \in I$

\end{enumerate}

To prove i., let $y\C 0$ (i.e. $y\in I$) and $x\leq y$.  Then
$x\cdot y = x$.  Since $\C$ is a congruence and $y\C 0$, we have
$x\cdot y \C x \cdot 0$.  Hence $x\C 0$.

%Then $x+y=y$, and $(x+y)\C 0$ ($\spadesuit$). Since $y\C 0$ and $\C$
%is a congruence, $\bar{y}\C 1$ ($\clubsuit$).  Again since $\C$ is a
%congruence, $[(x+y)\cdot \bar{y}]\C [0\cdot 1]$ by $(\spadesuit)$
%and $(\clubsuit)$, and $(x+y)\cdot \bar{y}=x$ so $x\C 0$, i.e. $x\in
%I$.

To prove ii., let $x\C 0$ and $y\C 0$.  Then $(x+y)\C (0+0)$.  To
prove iii., let $x\C 0$.  Note that $1\C 1$ since $\C$ is reflexive.
Then $(x;1)\C (0;1)$ and $0;1=0$, so $(x;1)\C 0$, and similarly for
$1;x$.  Also, since $x\C 0$, we have $\conv{x}\C \conv{0}$.  But
$\conv{0}=0$, so $\conv{x}\C 0$.  So $I$ is a relational ideal.

Now let $J=\{x\in A' : x\leq y\in I\}\subseteq A'$.  Now $J$ is a
relational ideal on $\A'$.  $J$ induces a congruence on $A'$, $\C' =
\{\langle x,y\rangle : x\vartriangle y \in J \}$.\footnote{Recall
that $\vartriangle$ denotes symmetric difference: $x\vartriangle y
:= x\cdot\bar{y} + \bar{x}\cdot y$} Now $\C=\C' \cap (A\times A)$:
the inclusion $\subseteq$ is clear. If $\langle x,y \rangle \in \C'
\cap (A\times A)$, then $ x,y  \in J$ and $x\vartriangle y \in J$.
But $J\cap A = I$, so $x\vartriangle y \in I$, and $\langle x,y
\rangle \in \C$. This concludes the proof.
\end{proof}

\begin{thm} $\hh\ss\{\Sub : E=E|E^{-1} \}=\ss\hh\{\Sub : E=E|E^{-1} \}$. \end{thm}

\begin{proof} $\Sub \in \textsf{RA}$.  Apply previous two
lemmas.\end{proof}

%%%%%%%%%%%%%%%%%%%%%%%%%%%%%%%%%%%%%%%%%%%%%%%%%%%%%%%%%%%%%%%%%%%%%%%%%%%%%%%%%%%%%%%%%%%%%%%%%%%%%%%%%%%%

Next we wish to show $\hh\{\Sub : E=E|E^{-1} \}\subseteq
\textsf{RRA}$.

\begin{defn} Let $E\neq\emptyset$ be an equivalence relation.  We
define the \emph{points of} $E$,

\[
\Pt:=\{p\subseteq E : E|p|E = E, \ p|E|p\subseteq \Id_E \}
\]
where $\Id_E = \textsf{Id}\cap E$.
\end{defn}

In the following lemma we have properties of the points of $E$
that we will need.  Note that i. provides a characterization of
the points.

\renewcommand{\theenumi}{\roman{enumi}}  %%%%%%%%%%%%%%%%%%%  Why does this need to be repeated here???
\renewcommand{\labelenumii}{\theenumii.}

\begin{lemma}
\begin{enumerate}
\item $p\in\Pt$ iff for all equivalence classes $U$ of $E$,
$\exists u\in U$ $ p\cap U^2 = \{\langle u,u \rangle\}$.

\item $p\in\Pt \Rightarrow p=p^{-1}\subseteq\Id_E$.

\item $R,S\subseteq E, \, p,q\in\Pt$; then
\begin{enumerate}
\item $E|p|R|q|E \, \cap \, E|p|S|q|E = E|p|(R\cap S)|q|E$

\item $E|R|p|E \, \cap \, E|p|S|E = E|R|p|S|E$

\item $E|p|R|q|E = E|q|R^{-1}|p|E$

\end{enumerate}

\item $\forall R\subseteq E \ \exists p,q\in \Pt \ E|R|E =
E|p|R|q|E$

\item $\forall R,S\subseteq E \ \exists p\in \Pt \ E|R|S|E =
E|R|p|S|E$

\end{enumerate}

\end{lemma}

\begin{proof} For this proof we will abbreviate ``$\langle u,v
\rangle \in p$" by ``$upv$". So the string ``$upvEupv$" indicates
that the following holds:

\begin{graph}(9,1)(-7,0)
\grapharrowlength{.2} \roundnode{u}(-4,0) \roundnode{v}(-2,0)
\roundnode{u2}(0,0) \roundnode{v2}(2,0) \diredge{u}{v}
\diredge{u2}{v2} \edge{v}{u2} \autonodetext{u}[s]{$u$}
\autonodetext{v}[s]{$v$} \autonodetext{u2}[s]{$u$}
\autonodetext{v2}[s]{$v$} \bowtext{u}{v}{0.1}{$p$}
\bowtext{v}{u2}{0.1}{$E$} \bowtext{u2}{v2}{0.1}{$p$}
\end{graph}

\bigskip

This will reduce somewhat the number of graphs that need to be
drawn.

\begin{enumerate}
\item $(\Rightarrow)$:  Assume $p\in\Pt$.  Let $U$ be a (nonempty)
equivalence class of $E$.  $E|p|E=E$ implies that $p\cap U^2$ is
nonempty.  So choose $\langle u,v \rangle \in p\cap U^2$.  Then
$upvEupv$ since $\langle u,v \rangle \in p$ and $u E v$.  But
$p|E|p\subseteq \Id_E$ by hypothesis, so $\langle u,v \rangle\in\Id$
and $u=v$.  Therefore $p\cap U^2 \subseteq \Id_E$.  \\
Now suppose that $\langle u,u \rangle , \, \langle v,v \rangle \in
p\cap U^2$.  Then $upuEvpv$, so $u(p|E|p)v$, but
$p|E|p\subseteq\Id$, so $u=v$.

$(\Leftarrow)$:  Suppose for all equivalence classes $U$, $ \exists
u \ p\cap U^2 = \{\langle u,u\rangle\}$.  Show $E|p|E=E$:  The
inclusion $\subseteq$ always holds.  To show $\supseteq$, let
$\langle x,y \rangle \in E$. Then there is some equivalence class
$U$ so that $x,y\in U$. We also have $p\cap U^2 = \{\langle u,u
\rangle \}$. Then
$xEupuEy$, and so $\langle x,y \rangle \in E|p|E$.  \\
Show $p|E|p\subseteq \Id_E$: Let $\langle x,y\rangle \in p|E|p$.
Then $\langle x,x \rangle, \langle y,y \rangle \in p$ and $xpxEypy$.
But since $xEy$, $x$ and $y$ are in the same equivalence class $U$.
So then $\langle x,y\rangle\in p\cap U^2$, which implies $x=y$.
Therefore $p|E|p\subseteq \Id_E$.

\bigskip
\hrule
\bigskip

\item follows from i.

\bigskip

\item \begin{enumerate}

\item We want $E|p|R|q|E \, \cap \, E|p|S|q|E = E|p|(R\cap
S)|q|E$.  So let $\langle u,v \rangle \in E|p|R|q|E \, \cap \,
E|p|S|q|E$.  Then there exists points (1)--(8) such
that\footnote{Edges labeled $p$ are now drawn undirected in light
of ii.}

\bigskip

\bigskip

\begin{graph}(12,5)(-7,-3)
\grapharrowlength{.2} \roundnode{u}(-6,0) \roundnode{1}(-4,1)
\roundnode{2}(-2,1) \roundnode{5}(0,1) \roundnode{6}(2,1)
\roundnode{v}(4,0) \roundnode{3}(-4,-1) \roundnode{4}(-2,-1)
\roundnode{7}(0,-1) \roundnode{8}(2,-1) \autonodetext{u}[w]{$u$}
\autonodetext{v}[e]{$v$} \autonodetext{1}[n]{(1)}
\autonodetext{2}[n]{(2)} \autonodetext{5}[n]{(5)}
\autonodetext{6}[n]{(6)} \autonodetext{3}[s]{(3)}
\autonodetext{4}[s]{(4)} \autonodetext{7}[s]{(7)}
\autonodetext{8}[s]{(8)} \edge{u}{1} \bowtext{u}{1}{0.13}{$E$}
\edge{u}{3} \bowtext{u}{3}{-.13}{$E$} \edge{1}{2}
\bowtext{1}{2}{.13}{$p$} \edge{3}{4} \bowtext{3}{4}{-.13}{$p$}
\diredge{2}{5} \bowtext{2}{5}{.13}{$R$} \diredge{4}{7}
\bowtext{4}{7}{-.13}{$S$} \edge{5}{6} \bowtext{5}{6}{.13}{$q$}
\edge{7}{8} \bowtext{7}{8}{-.13}{$q$} \edge{6}{v}
\bowtext{6}{v}{.13}{$E$} \edge{8}{v} \bowtext{8}{v}{-.13}{$E$}
\put(-3,0){\oval(3,4)} \put(1,0){\oval(3,4)}
\end{graph}

\bigskip

\bigskip

By i., (1),(2),(3),(4) are all the same point.  Likewise,
(5),(6),(7),(8) are all the same point.  So then $\langle
u,v\rangle\in E|p|(R\cap S)|q|E$.  \\

The proof of $\supseteq$ is trivial by $|$-monotonicity.

\bigskip
\hrule
\bigskip

\item The proof of $E|R|p|E \, \cap \, E|p|S|E = E|R|p|S|E$ is
similar to the previous.

\bigskip
\hrule
\bigskip

\item $E|p|R|q|E = E|q|R^{-1}|p|E$:

Let $\langle u,v\rangle \in E|p|R|q|E$:

\smallskip

\begin{graph}(12,3)(-7,-1.5)
\grapharrowlength{.2} \roundnode{u}(-6,0) \roundnode{1}(-4,0)
\roundnode{2}(-2,0) \roundnode{3}(0,0) \roundnode{4}(2,0)
\roundnode{v}(4,0) \autonodetext{u}[w]{$u$}
\autonodetext{v}[e]{$v$} \edge{u}{1} \edge{1}{2} \diredge{2}{3}
\edge{3}{4} \edge{4}{v} \bowtext{u}{1}{-.13}{$E$}
\bowtext{1}{2}{-.13}{$p$} \bowtext{2}{3}{-.13}{$R$}
\bowtext{3}{4}{-.13}{$q$} \bowtext{4}{v}{-.13}{$E$}
\end{graph}

\bigskip

\bigskip

This gives

\bigskip

\bigskip

\begin{graph}(8,3)(-7,-1.5)
\grapharrowlength{.2}  \roundnode{u}(-4,0) \roundnode{1}(-2,0)
\roundnode{2}(0,0) \roundnode{v}(2,0) \autonodetext{u}[w]{$u$}
\autonodetext{v}[e]{$v$} \autonodetext{1}[s]{$x$}
\autonodetext{2}[s]{$y$} \edge{u}{1} \diredge{1}{2} \edge{2}{v}
\bowtext{u}{1}{-.13}{$E$} \bowtext{1}{2}{-.13}{$R$}
\loopedge{1}{25}(.1,.4) \loopedge{2}{25}(.1,.4)
\bowtext{u}{2}{.22}{$p$} \bowtext{1}{v}{.22}{$q$}
\bowtext{2}{v}{-.13}{$E$}
\end{graph}

\bigskip

\bigskip

Now $uEy$ and $vEx$, so we can write $uEyqyR^{-1}xpxEv$, and
$\langle u,v\rangle\in E|q|R^{-1}|P|E$.  The inclusion $\supseteq$
is similar.

%But since $u,v,(1),(2)$ are all in the same equivalence class, (1)
%and (2) are the same point, and so we have

%\bigskip

%\begin{graph}(6,3)(-7,-1.5)
%\grapharrowlength{.2}  \roundnode{u}(-2,0) \roundnode{1}(0,0)
% \roundnode{v}(2,0) \autonodetext{u}[w]{$u$}
%\autonodetext{v}[e]{$v$} \edge{u}{1} \edge{1}{v}
%\bowtext{u}{1}{-.13}{$E$} \loopedge{1}{20}(.12,.4)
%\loopedge{1}{20}(-.12,.4) \loopedge{1}{20}(0,-.4)
%\bowtext{u}{v}{.22}{$p \ \ q$} \bowtext{u}{v}{-.24}{$R$}
%\bowtext{1}{v}{-.13}{$E$}
%\end{graph}

%\bigskip

%\bigskip

So then $E|p|R|q|E = E|q|R^{-1}|p|E$.
\end{enumerate}

\bigskip
\hrule
\bigskip

\item Let $\{U_\alpha\}_{\alpha \in I}$ be the equivalence classes
of $E$.  Let $R\subseteq E$.  Let $R_\alpha = R\cap U_\alpha^2$. For
all non-empty $R_\alpha$, pick $\langle u_\alpha, v_\alpha \rangle
\in R_\alpha$. For $R_\alpha$ empty, let $\langle u_\alpha,v_\alpha
\rangle\in U_\alpha$.  Let $p:= \{\langle u_\alpha, u_\alpha \rangle
: \alpha \in I \} \ q:= \{\langle v_\alpha, v_\alpha \rangle :
\alpha \in I \}$. Then $E|R|E=E|p|R|q|E$.
\bigskip
\hrule
\bigskip

\item Let $R,S \subseteq E$.  Let $\{U_\alpha\}_{\alpha \in I}$ be
as above.  When $(R|S)\cap U_\alpha^2 \neq  \emptyset$, pick
$\langle x,y \rangle \in (R|S)\cap U_\alpha^2$.  For every such
alpha, $\exists u_\alpha\in U_\alpha$, $xRu_\alpha S y$.  Let $p:=
\{\langle u_\alpha, u_\alpha \rangle : \alpha \in I \}$.   Then
$E|R|S|E=E|R|p|S|E$.

\end{enumerate}
\end{proof}

%%%%%%%%%%%%%%%%%%%%%%%%%%%%%%%%%%%%%%%%%%%%%%%%%%%%%%%%%%%%%%%%%%%%%%%%%%%%%%%%%%%%%%%

\begin{lemma}
Let $\B = \langle B,+, \, -, \, ;, \, \conv{}, \, \1 \rangle$ be a
nondegenerate algebra of relational type.  Let $h:\Sub
\twoheadrightarrow \B$ be a homomorphism onto $\B$ with maximal
kernel. Define $\sigma : \Sub \longrightarrow \RR(\Pt)$ by

\[
\sigma(R) = \{\langle p,q \rangle \in \Pt \times \Pt :
h(E)=h(E|p|R|q|E)\}
\]

Then
\begin{enumerate}
\item $\sigma(\emptyset)=\emptyset$

\item $\sigma(E) = \Pt \times \Pt$

\item $R\subseteq S \Rightarrow \sigma(R)\subseteq \sigma(S)$

\item $\sigma (R\cup S)=\sigma(R)\cup \sigma(S)$

\item $\sigma (E\setminus R)\cap \sigma(R) = \emptyset$

\item $\sigma(E\setminus R) \cup \sigma(R) = \Pt \times \Pt$, and
consequently $\sigma(E\setminus R) = \sigma(E)\setminus \sigma(R)$

\item $\sigma (R|S) = \sigma(R)|\sigma(S)$

\item $\sigma(R^{-1})= \sigma(R)^{-1}$

\item $\sigma(\Id_E) \supseteq \textsf{Id}\cap (\Pt \times \Pt)$
\end{enumerate}

i.--ix. say that $\sigma$ is almost a homomorphism; it would be if
equality held in ix.  We will call a function that satisfies
i.--ix. a \emph{near--homomorphism}.  \\

Note: the maximality of the kernel of $h$ is needed only for iv.
and vi.
\end{lemma}

\begin{proof}

\begin{enumerate}

\item $\sigma (\emptyset)=\{ \langle p,q \rangle :
h(E)=h(E|p|\emptyset|q|E\}=\emptyset$, since $h(E)\neq
h(\emptyset)$.

\item $E=E|p|E|q|E$ for all $p,q$, so $h(E)=h(E|p|E|q|E)$, and ii.
holds.

%\stepcounter{enumi}

%%%%%%%%%%%%%%%%%%%%%%%%%%%%%%%%%%%%%%%%%%%%%%%%%%%%%%%%%%%%%%%%%%%%%%%%%%%% iii
\item $R\subseteq S \Rightarrow \sigma(R)\subseteq \sigma(S)$:\\

Let $R\subseteq S$, so $S\cap R=R$.  Let $\langle p,q \rangle \in
\sigma(R)$.  Then $h(E)=h(E|p|R|q|E)$, and so

\begin{align*}
h(E|p|S|q|E) &= h(E|p|S|q|E \, \cap \, E) \\
             &= h(E|p|S|q|E) \cdot h(E)  &&\text{($h$ a hom.)} \\
             &= h(E|p|S|q|E) \cdot h(E|p|R|q|E) && p,q\in \sigma(R) \\
             &= h(E|p|S|q|E \, \cap \, E|p|R|q|E) &&\text{($h$ a hom.)} \\
             &= h(E|p|(R\cap S)|q|E) &&\text{(by 3.2.6)} \\
             &= h(E|p|R|q|E)  &&(S\cap R=R)  \\
             &=h(E) && p,q\in \sigma(R)
\end{align*}

\bigskip
\hrule
\bigskip
%%%%%%%%%%%%%%%%%%%%%%%%%%%%%%%%%%%%%%%%%%%%%%%%%%%%%%%%%%%%%%%%%%%%%%%%%%%%% iv
\item $\sigma (R\cup S)=\sigma(R)\cup \sigma(S)$:\\

From iii. we get $\sigma(R), \sigma(S) \subseteq \sigma(R\cup S)$,
so $\supseteq$ holds.  \\

For $\subseteq$, let $\langle p,q \rangle \in \sigma(R\cup S)$.
Then

\begin{align*}
h(E) &= h(E|p|(R\cup S)|q|E) \\
     &= h(E|p|R|q|E \, \cup \, E|p|S|q|E) && (|-\text{dist.}) \\
     &= h(E|p|R|q|E) + h(E|p|S|q|E) && \text{($h$ a hom.)}
\end{align*}

Since the kernel of $h$ is maximal, any image of $\Sub$ under $g$
(namely $\B$) is simple by general algebraic
considerations.\footnote{If $\A$ is an algebra and $\theta$ is a
congruence on $\A$, then there is a 1-1 correspondence between
congruences $\theta'\supseteq\theta$ and congruences on $\A/\theta$.
 Therefore if $\theta$ is maximal then $\A/\theta$ is simple.} For all $X\in\sub$, either $h(E|X|E)=h(E)$ or $h(E|X|E)=h(\emptyset)$.
Since $h(E)\neq h(\emptyset)$, it is not the case that $h(E|p|R|q|E)
= h(E|p|S|q|E) = h(\emptyset)$.  So at least one of $h(E|p|R|q|E), \
h(E|p|S|q|E)$ is equal to $h(E)$.  Therefore $\langle p,q \rangle$
is in either $\sigma(R)$ or $\sigma(S)$, hence in their union.

\bigskip
\hrule
\bigskip
%%%%%%%%%%%%%%%%%%%%%%%%%%%%%%%%%%%%%%%%%%%%%%%%%%%%%%%%%%%%%%%%%%%%%%%%%%%%%%%%% v
\item $\sigma (E\setminus R)\cap \sigma(R) = \emptyset$:\\

Suppose both $\sigma (E\setminus R)$, $\sigma(R)$ are nonempty.
(Otherwise there is nothing to show.)  Take $\langle p,q \rangle
\in \sigma(E\setminus R)$.  So $h(E) = h(E|p|(E\setminus r)|q|E)$.
Then

\begin{align*}
h(E|p|R|q|E) &= h(E|p|R|q|E) \cdot h(E) \\
     &= h(E|p|R|q|E) \cdot h(E|p|(E\setminus R)|q|E) &&\text{(hypothesis)} \\
     &= h(E|p|R|q|E \cap E|p|(E\setminus R)|q|E) && \text{($h$ a hom.)} \\
     &= h(E|p|(R \cap (E\setminus R))|q|E) && \text{(3.2.6)} \\
     &= h(E|p|\emptyset|q|E) \\
     &= h(\emptyset)\neq h(E)
\end{align*}

So $\langle p,q \rangle \notin \sigma(R)$.  Thus
$\sigma(E\setminus R) \subseteq \Pt^2 \setminus \sigma(R)$, and
the intersection is empty.

\bigskip
\hrule
\bigskip
%%%%%%%%%%%%%%%%%%%%%%%%%%%%%%%%%%%%%%%%%%%%%%%%%%%%%%%%%%%%%%%%%%%%%%%%%%%%%%%%%% vi
\item $\sigma(E\setminus R) \cup \sigma(R) = \Pt \times \Pt$, and
consequently $\sigma(E\setminus R) = \sigma(E)\setminus \sigma(R)$:\\

$\sigma(E\setminus R) \cup \sigma(R) = \Pt^2$ follows from iv.
From $\sigma (E\setminus R)\cap \sigma(R) = \emptyset$ and
$\sigma(E\setminus R) \cup \sigma(R) = \Pt^2$ it follows that
$\sigma(E\setminus R)$ is the boolean complement of $\sigma(R)$,
hence $\sigma(E\setminus R) = \sigma(E)\setminus \sigma(R)$.

\bigskip
\hrule
\bigskip
%%%%%%%%%%%%%%%%%%%%%%%%%%%%%%%%%%%%%%%%%%%%%%%%%%%%%%%%%%%%%%%%%%%%%%%%%%% vii
\item $\sigma (R|S) = \sigma(R)|\sigma(S)$:\\

Let $\langle p,q \rangle \in \sigma(R|S)$.  So
$h(E)=h(E|p|R|S|q|E)$.  By v. of lemma 3.2.6, $\exists \, r \in\Pt \
E|p|R|S|q|E=E|p|R|r|S|q|E$.  Therefore $E|p|R|S|q|E=E|p|R|r|S|q|E
\subseteq E|p|R|r|E \subseteq E$ (since $S|q|E\subseteq E$), and
$E|p|R|S|q|E=E|p|R|r|S|q|E \subseteq E|r|S|q|E \subseteq E$ (since
$E|p|R \subseteq E$).  So then

\begin{align*}
h(E) &= h(E|p|R|S|q|E) \\
     &= h(E|p|R|S|q|E \, \cap \, E|p|R|r|E) && 3.2.6\\
     &= h(E|p|R|S|q|E)\cdot h(E|p|R|r|E) && h\text{ a hom.}\\
     &= h(E) \cap h(E|p|R|r|E) &&\text{hyp.} \\
     &= h(E|p|R|r|E) && h\text{ a hom.}
\end{align*}

And so $\langle p,r \rangle \in \sigma(R)$. Similarly, $\langle
r,q \rangle \in \sigma(S)$.  So $\langle p,q
\rangle \in \sigma(R)|\sigma(S)$.\\

Conversely, let $\langle p,q \rangle \in \sigma(R)|\sigma(S)$.
Then $\exists \, r \in \Pt \ \langle p,r \rangle \in \sigma(R), \
\langle r,q \rangle \in \sigma(S)$.  So
$h(E|p|R|r|E)=h(E)=h(E|r|S|q|E)$.  Then

\begin{align*}
h(E) &= h(E|E) \\
     &= h(E);h(E) && h\text{ a hom.}\\
     &= h(E|p|R|r|E);h(E|r|S|q|E) && \text{hyp.}\\
     &= h(E|p|R|r|E|E|r|S|q|E) && h\text{ a hom.}\\
     &= h(E|p|R|r|S|q|E) &&(r|E|E|r = r|E|r = r, \ r\in\Pt) \\
     &= h(E|p|R|S|q|E) && \text{hyp.}
\end{align*}

%So $\langle p,q \rangle \in \sigma(R|r|S)$.  Now $R|r|S \subseteq
%R|\Id_E |S =R|S$, and so $\sigma(R|r|S)\subseteq \sigma(R|S)$ by
%monotonicity of $\sigma$ (iii. above).
Hence $\langle p,q \rangle
\in \sigma(R|S)$.

\bigskip
\hrule
\bigskip
%%%%%%%%%%%%%%%%%%%%%%%%%%%%%%%%%%%%%%%%%%%%%%%%%%%%%%%%%%%%%%%%%%%%%%%%%%%% viii
\item $\sigma(R^{-1})= \sigma(R)^{-1}$:
\begin{align*}
\langle q,p \rangle \in \sigma(R^{-1}) &\iff
h(E)=h(E|q|R^{-1}|p|E) = h(E|p|R|q|E) \\
&\iff \langle p,q \rangle \in \sigma(R) \\
&\iff \langle q,p \rangle \in \sigma(R)^{-1}
\end{align*}

\bigskip
\hrule
\bigskip
%%%%%%%%%%%%%%%%%%%%%%%%%%%%%%%%%%%%%%%%%%%%%%%%%%%%%%%%%%%%%%%%%%%%%%%%%%%%%% ix
\item $\sigma(\Id_E) \supseteq \textsf{Id}\cap (\Pt \times \Pt)$:

If $\langle p,p \rangle \in \textsf{Id}\cap \Pt^2$, then
$h(E)=h(E|p|\Id_E|p|E)$ since $E=E|p|E = E|p|\Id_E|E =
E|p|\Id_E|(E|p|E) = E|p|\Id_E|p|E$.  So then $\langle p,p \rangle
\in \sigma(\Id_E)$.  \qedhere

\end{enumerate}
\end{proof}

\begin{lemma}
Let $\A = \langle A, +, \, -, \, ;, \, \conv{}, \, \1 \rangle$ with
$\A\models \1\,\conv{}=\1 \meet x;\1=\1;x=x$.  Let $g:A
\longrightarrow \rr(U)$ be a near-homomorphism for some set $U$.
Then $g(\1)$ is an equivalence relation and $h:\A \longrightarrow
\RR\big(U / g(\1)\big)$  given by $h(a) = \{ \big\langle r / g(\1) ,
s / g(\1) \big\rangle : \langle r,s \rangle \in g(a) \}$ is a
homomorphism. Furthermore, $h$ is injective if $g$ is.
\end{lemma}

\begin{proof}
First we show that $g(\1)$ is an equivalence relation:
$g(\1)|g(\1)^{-1} = g(\1)|g(\1\,\conv{}) = g(\1)|g(\1) = g(\1;\1) =
g(\1)$, hence $g(\1)$ is an equivalence relation over its field.
From now on we will denote $g(\1)$ by $E$.  \\
Now we show that $h$ is a homomorphism.

$h(a+b) = h(a) \cup h(b)$:
\begin{align*}
\langle r/E, s/E \rangle \in h(a+b) &\iff \langle r,s \rangle \in
g(a+b) = g(a)\cup g(b) \\
         &\iff \langle r,s \rangle \in g(a) \text{ OR } \langle r,s \rangle \in g(b) \\
         &\iff \langle r/E, s/E \rangle \in h(a) \text{ OR } \langle r/E, s/E \rangle \in h(b) \\
         &\iff \langle r/E, s/E \rangle \in h(a) \cup h(b)
\end{align*}

$h(\bar{a}) = (U/E)^2\setminus h(a)$:
\begin{align*}
\langle r/E, s/E \rangle \in h(\bar{a}) &\iff \langle r,s \rangle
\in
g(\bar{a}) = U^2\setminus g(a) \\
         &\iff \langle r,s \rangle \notin g(a) \text{ AND } r,s \in U\\
         &\iff \langle r/E, s/E \rangle \notin h(a) \text{ AND } r/E,s/E \in U/E \\
         &\iff \langle r/E, s/E \rangle \in (U/E)^2\setminus h(a)
\end{align*}

$h(a;b)=h(a)|h(b)$:
\begin{align*}
\langle r/E, s/E \rangle \in h(a;b) &\iff \langle r,s \rangle \in
g(a;b) = g(a)|g(b) \\
         &\iff \exists t \ \langle r,t \rangle \in g(a) \text{ AND } \langle t,s \rangle \in g(b) \\
         &\iff \exists t \ \langle r/E, t/E \rangle \in h(a) \text{ AND } \langle t/E, s/E \rangle \in h(b) \\
         &\iff \langle r/E, s/E \rangle \in h(a)|h(b)
\end{align*}

$h(\conv{a}) = h(a)^{-1}$:
\begin{align*}
\langle r/E, s/E \rangle \in h(\conv{a}) &\iff \langle r,s \rangle
\in
g(\conv{a}) = g(a)^{-1} \\
         &\iff \langle s,r \rangle \in g(a)  \\
         &\iff \langle s/E, r/E \rangle \in h(a) \\
         &\iff \langle r/E, s/E \rangle \in h(a)^{-1}
\end{align*}

$h(\1)=\textsf{Id}\cap (U/E)^2$:
\begin{align*}
\langle r/E, s/E \rangle \in h(\1) &\iff \langle r,s \rangle \in
g(\1) = E \text{, \footnotesize{(an equivalence relation)}}\\
         &\iff r E s \\
         &\iff  r/E = s/E \\
         &\iff \langle r/E, s/E \rangle \in \textsf{Id}\cap (U/E)^2
\end{align*}

Now we suppose that $g$ is 1-1, and prove that $h$ is also.  Let
$a,b\in A$, $a\neq b$.  Then $g(a)\neq g(b)$.  Suppose without loss
of generality that $g(a)\setminus g(b)\neq\emptyset$.  We want to
show that $h(a)\neq h(b)$.  It will suffice to show that if $\langle
x,y\rangle \in g(a)\setminus g(b)$, then $\langle x/E,y/E \rangle
\in h(a)\setminus h(b)$.  So let $\langle x,y\rangle \in
g(a)\setminus g(b)$.  We want $\langle x/E,y/E \rangle$ to be
distinct from all $\langle r/E,s/E \rangle$, where $\langle r,s
\rangle\in g(b)$.  Suppose by way of contradiction that there is
some $\langle r,s \rangle \in g(b)$ so that $\langle x/E,y/E \rangle
=\langle r/E,s/E \rangle$.  Then $xEr$ and $yEs$.  So we have

\bigskip

\begin{graph}(6,4)(-6,-2)
\grapharrowlength{.2} \roundnode{x}(-2,0) \roundnode{r}(0,0)
\roundnode{y}(-2,-2) \roundnode{s}(0,-2) \edge{x}{r} \edge{y}{s}
\diredge{x}{y} \diredge{r}{s} \autonodetext{x}[nw]{$x$}
\autonodetext{r}[ne]{$r$} \autonodetext{y}[sw]{$y$}
\autonodetext{s}[se]{$s$} \bowtext{x}{r}{0.12}{$E$}
\bowtext{y}{s}{-0.12}{$E$} \bowtext{x}{y}{-0.22}{$g(a)$}
\bowtext{r}{s}{0.22}{$g(b)$}
\end{graph}

\bigskip

So then
\begin{align*}
\langle x,y \rangle \in E|g(b)|E &= g(\1)|g(b)|g(\1) && E=g(\1) \\
                                 &= g(\1;b;\1) && g \text{ a hom.}
                                 \\
                                 &= g(b)
\end{align*}

This stands in contradiction to the assumption that $\langle
x,y\rangle \in g(a)\setminus g(b)$.  Therefore $\langle x/E,y/E
\rangle \in h(a)\setminus h(b)$, and $h(a)\neq h(b)$, and so $h$ is
1-1 also. \qedhere
\end{proof}

\begin{lemma} Let $g:\Sub \twoheadrightarrow \B$ be a homomorphism
onto a non-degenerate algebra $\B$ such that $g$ has a maximal
kernel.  Then $\B\in\textsf{RRA}$.
\end{lemma}

\begin{proof} Let $\sigma: \sub \longrightarrow \rr(\Pt)$, $\sigma(R)=\{\langle p,q\rangle \in \Pt:g(E)=g(E|p|R|q|E)\}$.   Consider $g^{-1}|\sigma \subseteq B \times \rr(\Pt)$.

\[
\begin{diagram}
\node{\Sub} \arrow{s,l,A}{g} \arrow{e,t,A}{\sigma} \node{\RR(\Pt)}  \\
  \node{\B} \arrow{ne,r,..}{g^{-1}|\sigma} \\
\end{diagram}
\]

Note that since $g$ is surjective, the domain of $g^{-1}|\sigma$
is all of $B$.  It is easy to check that $g^{-1}|\sigma$ is
functional and is a near-homomorphism (just use the definition of
$\sigma$).  Then by the previous lemma, there is an $f: \B
\longrightarrow \RR\big(\Pt / g(\Id_E)\big)$ that is a
homomorphism, and $f$ is 1-1 if $g^{-1}|\sigma$ is 1-1.

$g^{-1}|\sigma$ is 1-1:  Let $b,c \in B, \ b\neq c$.  Let $R
\stackrel{g}{\longmapsto} b$, $S \stackrel{g}{\longmapsto} c$.
$R\neq S$, so either $R\cap (E\setminus S) \neq\emptyset$ or $S\cap
(E\setminus R) \neq\emptyset$.  Suppose that $R\cap (E\setminus S)
\neq\emptyset$. Choose $p,q\in \Pt, \ E=E|p|[R\cap(E\setminus
S)]|q|E$.  Then

\begin{align*}
g(E) &= g(E|p|R|q|E \, \cap \, E|p|(E\setminus S)|q|E) \\
     &= g(E|p|R|q|E)\cdot g(E|p|(E\setminus S)|q|E)
\end{align*}

Hence $g(E|p|R|q|E)=g(E)=g(E|p|(E\setminus S)|q|E)$, and so
$\langle p,q \rangle \in \sigma(R)\setminus\sigma(S) =
g^{-1}|\sigma(b) \setminus g^{-1}|\sigma(c)$.  Therefore
$g^{-1}|\sigma(b) \setminus g^{-1}|\sigma(c) \neq \emptyset$, and
hence $g^{-1}|\sigma(b) \neq g^{-1}|\sigma(c)$.

We then conclude that $f$ is 1-1 also; so $f$ embeds $\B$ into a
square relation algebra:  $\B\iso|\subseteq \RR\big(\Pt /
g(\Id_E)\big)$. \qedhere

\end{proof}

\begin{thm}  $\hh\{\Sub : E=E|E^{-1} \}\subseteq\textsf{RRA}$. \end{thm}

\begin{proof}
Let $\B \in \hh\{\Sub : E=E|E^{-1} \}$.  Then there is some
homomorphism $g:\Sub \twoheadrightarrow \B$ for some $E$.  Let $I=
\text{ker }g = g^{-1}[0]$.  I is a relational ideal.  Let $b,c \in
B, \ b\neq c$.  Then $\exists R,S \subseteq E, \
R\stackrel{g}{\longmapsto} b , \ S\stackrel{g}{\longmapsto} c$. Let
$T:=\overline{E|(R\vartriangle S)|E}$.

If $T\in I$, then extend $I$ to a maximal relational ideal $J$ (use
Zorn's Lemma).  If $T\notin I$, then define $I'=\{X\subseteq E :
\exists X_1 \in I, \ X\subseteq X_1  \cup  E|T|E \}$.  It is
straightforward to check that $I'$ is a relational ideal containing
$I$.  To see that $I'$ is proper, suppose the contrary, so that
$E\in I'$. Then there is some $X_1\in I$ so that $E=X_1 \cup E|T|E$.
Then $X_1 \supseteq \overline{E|T|E}$, and so

\begin{align*}
  X_1 &\supseteq \overline{E|T|E} \\
      &= \overline{E|\overline{E|(R\vartriangle S)|E}|E} &&\text{def. of $T$} \\
      &= E|(R\vartriangle S)|E
      &&\overline{E\Big|\overline{E|X|E}\Big|E}=E|X|E
\end{align*}

But $X_1\in I$, and since $E|(R\vartriangle S)|E \subseteq X_1$, we
have that $E|(R\vartriangle S)|E\in I$ also, and consequently that
$R\vartriangle S \in I$.  But that means that $g(R)=g(S)$, contrary
to assumption.  Therefore $I'$ is proper.  Since $I'$ is proper, it
is included in a maximal relational ideal $J$ (use Zorn's Lemma).
Thus whether or not $T\in I$, we get a maximal relational ideal
$J\supseteq I\cup\{T\}$.

 Then $\exists h_J:\Sub
\twoheadrightarrow \Sub / J$ with (maximal) kernel $J$.  By the
previous lemma, $\Sub / J$ is isomorphic to a square proper relation
algebra $\Q_{b,c} \subseteq \RR (U)$.  Then $g^{-1}|h_J:\B
\longrightarrow \Q_{b,c}$ is a homomorphism that separates $b,c$:

Suppose for $R,S \in \sub$, $g(R)=g(S)$.  Then $g(R\vartriangle
S)=0$. So $R\vartriangle S \in I\subseteq J$.  Now $R\vartriangle S
\in J$, so $h_J(R\vartriangle S) =0$, and $h_J(R) = h_J(S)$.  So
$g^{-1}|h_J$ is functional.  It is straightforward to check that
$g^{-1}|h_J$ is a homomorphism.  To show that $g^{-1}|h_J$ separates
$b,c$, recall that $g(R)=b, \ g(S)=c$, and $J\supseteq I\cup\{T\}$.
Thus

\begin{align*}
T=\overline{E|(R\vartriangle S)|E}\in J &\Longrightarrow
\overline{E|(R\vartriangle S)|E} \stackrel{h_J}{\longmapsto} 0 \\
       &\Longrightarrow E|(R\vartriangle S)|E \stackrel{h_J}{\longmapsto}
       1 \\
       &\Longrightarrow \text{ it is not the case that } (R\vartriangle S \stackrel{h_J}{\longmapsto}
       0) \\
       &\Longrightarrow R\vartriangle S \notin J \\
       &\Longrightarrow h_J(R)\neq h_J(S)
\end{align*}

So for each $b\neq c$ we get a separating homomorphism
$g^{-1}|h_J$ to a square proper relation algebra.  Thus we have a
homomorphism

\[
h:\B \longrightarrow \prod_{\substack{b,c\in\B
\\ b\neq c}} \Q_{b,c}
\]

given by $h(x) = \langle g^{-1}|h_J(x) : b,c \in \B, \ b\neq c
\rangle$. $h$ is an embedding into a product of proper relations
algebras. Hence $\B \in \ss\pp\textsf{RRA} = \textsf{RRA}$. \qedhere
\end{proof}

\begin{thm} $\hh\textsf{RRA}=\textsf{RRA}$. \end{thm}

\begin{proof}

$\B \in \hh\textsf{RRA}$.  $\exists \A \in \textsf{RRA}, \ \B\in
\hh\{\A\}$.  Now $\A \iso|\subseteq\Sub$ for some $E$, and so
$\B\in \hh\ss\{\Sub : E=E|E^{-1} \}$.  So then $\B \in
\hh\ss\{\Sub: E=E|E^{-1}\} = \ss\hh\{\Sub : E=E|E^{-1} \}
\subseteq \ss\textsf{RRA} = \textsf{RRA}$. \qedhere
\end{proof}

We now conclude that
$\textsf{RRA}=\hh\textsf{RRA}=\ss\textsf{RRA}=\pp\textsf{RRA}$, and
consequently that $\textsf{RRA}$ is definable by equations by
Birkhoff's theorem (see \cite{Birk}, the appendix).

%\renewcommand{\theenumi}{\arabic{enumi}}
%\renewcommand{\labelenumi}{[\theenumi]}

%auto-ignore
\chapter{Other axiomatizations of RRA}

In the previous chapter we saw that \textsf{RRA} has an equational
axiomatization.  Now we ask whether \textsf{RRA} is axiomatizable
by finitely many equations, or finitely many first order
sentences, or by infinitely many equations but using only finitely
many variables.  The answer to each will be ``no.''  First we
establish the existence of a countable set of finite algebras with
special properties.

\section{Relation algebras and projective geometries}

Both Lyndon and J\'{o}nsson developed connections between
projective geometry and relation algebra.  See \cite{Jon59} and
\cite{Lyn61}. The following definition is from \cite{Lyn61}.

\begin{defn}
A  \emph{(projective) geometry} is a set a points $P$ and a set a
lines $L$ such that $\ell \in L \Rightarrow \ell \subseteq P$ and
satisfying

\begin{enumerate}
\item $L\neq \emptyset$, and $\forall \ell \in L, \ |\ell| \geq
4$.

\item $\forall p,q\in P, \ p\neq q, \ \exists ! \ell \in L \
p,q\in\ell$.  We write $\ell = \overline{pq}$.

\item if $p,q,r\in P, \ p\neq q \neq r \neq p$, and $\exists \ell\in
L \ \ell \cap \overline{pq} \neq \emptyset, \ \ell\cap
\overline{pr}\neq \emptyset$, but $(\ell \cap
\overline{pq})\cap(\ell \cap \overline{pr})=\emptyset$, then $\ell
\cap \overline{qr} \neq\emptyset$.
\end{enumerate}
A \emph{projective line} is a geometry such that $|L|=1$.

\noindent A \emph{projective plane} is a geometry such that every
line contains $n+1$ points, and every point lies on $n+1$ lines.
$n$ is said to be the \emph{order} of the plane.
\end{defn}

We are interested in building relation algebras from projective
lines.  Given a finite projective line $G(P,L)$, Let $\imath\notin
P$, and let $\A(G) = \langle \text{Sb}(P\cup\{\imath\}), \cup, \, -,
\, ;,\, \conv{}, \, \{\imath\} \rangle$, where conversion is
identification ($\conv{x}=x$), and ; is given on the atoms ($\imath$
and all $p\in P$) by $\{p\};\{p\}=\{p\}\cup \{\imath\}$ and for
$p\neq q$, $\ds \{p\};\{q\} = \{r : p\neq r \neq q\} $.

Now we wish to give a more abstract (and more general) definition.
Let $\gamma\subseteq \{1,2,3\}$.  Then $\underline{E}_\alpha^\gamma$
is a complete atomic symmetric integral relation algebra on $\alpha$
atoms. For finite algebras, we write $\underline{E}_{n+1}^\gamma$,
$\alpha = n+1$.  In this case there are $n+1$ atoms and $n$
diversity atoms.  When $\alpha$ is infinite we leave off the ``+1".
Relative multiplication is given on the atoms by
\[
a;a = \sum\{ c\in \text{At} \underline{E}_\alpha^\gamma :
|\{a,c\}|\in \gamma \}\cup\{\1\}
\]
and for $a\neq b$,
\[
a;b = \sum\{ c\in \text{At} \underline{E}_\alpha^\gamma :
|\{a,b,c\}|\in \gamma \}
\]

If $\gamma=\{1,3\}$, then  $\underline{E}_{n+1}^\gamma$ is identical
to the algebra defined above, the \emph{Lyndon algebra of a
projective line of order $n-1$}.

\begin{thm}[Lyndon, '61]   $\E \in \textsf{RRA}$ iff there exists a projective plane of order $n-1$.  (See \cite{Lyn61}.) \end{thm}

\begin{thm}[Bruck-Ryser '49] There exist infinitely many integers
such that there is no projective plane of that order.  (See
\cite{BR}.)
\end{thm}

By these two theorems we establish the existence of a countable
set of arbitrarily large finite non-representable relation
algebras. This collection will be central to the proofs of the
theorems of both Monk and J\'{o}nsson.

\section{Non-finite axiomatizability}

That \textsf{RRA} is not finitely axiomatizable is a theorem due to
Monk.  The key to the proof is the standard model-theoretic
ultraproduct construction.

\begin{lemma} Let $\{A_i\}_{i\in \omega}$ be a collection of
finite sets such that $|A_n|\geq 2^n$.  Let $\F$ be a
non-principal ultrafilter on $\omega$.  Then
\[
\left|\prod_{i\in\omega} A_i \bigg/ \F \right|  = 2^{\aleph_0}
\]

\end{lemma}

\begin{proof}
We embed $2^\omega$ in the ultraproduct.  This is sufficient since
$|\prod A_i|=2^{\aleph_0}$, so we only need to establish that the
ultraproduct is at least that big.

Let $\alpha \in 2^\omega$.  We will construct $f_\alpha \in \prod
A_i$ such that for $\alpha\neq \beta$, we will have $f_\alpha
\not\sim_\F f_\beta$.  Thus $f_\alpha$ and $f_\beta$ will be in
distinct equivalence classes in the ultraproduct.

So we want $f_\alpha : \omega \longrightarrow \cup A_i$, so that
$f_\alpha (n) \in A_n$.  We denote the elements of $A_i$ as
follows: $\{a_i^0, a_i^1, a_i^2, \ldots, a_i^{2^i - 1} \}
\subseteq A_i$, since $A_i$ has at least $2^i$ elements.  Then for
any $\alpha$, let $f_\alpha(0)=a_0^0$,  the one element guaranteed
to be in $A_0$.  Let
\[
f_\alpha(1) =
\begin{cases}
a_1^1, & \alpha(1)=1 \\
a_1^0, & \alpha(1)=0
\end{cases}
\]
So $f_\alpha(1)\in A_1$. Let
\[
f_\alpha(2) =
\begin{cases}
a_2^3, & \alpha(1)=1 \meet \alpha(2)= 1 \\
a_2^2, & \alpha(1)=0 \meet \alpha(2)= 1 \\
a_2^1, & \alpha(1)=1 \meet \alpha(2)= 0 \\
a_2^0, & \alpha(1)=0 \meet \alpha(2)= 0
\end{cases}
\]
So $f_\alpha(2)\in A_2$.  In general, let $f_\alpha(n)=a_n^k$,
where $0\leq k \leq 2^n - 1$, and where $k$ is the integer given
by the binary digits $\alpha(1) \underline{\qquad}\alpha(n)$ read
from left to right.  For example, if $\alpha = 10110$ (after
dropping the first digit $\alpha(0)$), then
\begin{align*}
f_\alpha (1) &= a_1^1 \in A_1 \\
f_\alpha (2) &= a_2^1 \in A_2 \\
f_\alpha (3) &= a_3^5 \in A_3 \\
f_\alpha (4) &= a_4^{13} \in A_4 \\
f_\alpha (5) &= a_5^{13} \in A_5 \\
             &\vdots
\end{align*}
The following binary tree depicts the ``path" of $f_{10110\ldots}$
through the algebras $A_0, A_1, \ldots$
\pagebreak[2]
\begin{center}
{\footnotesize (If $\alpha(i)=1$, ``go up."  If $\alpha(i) = 0$,
``go down.")}
\end{center}

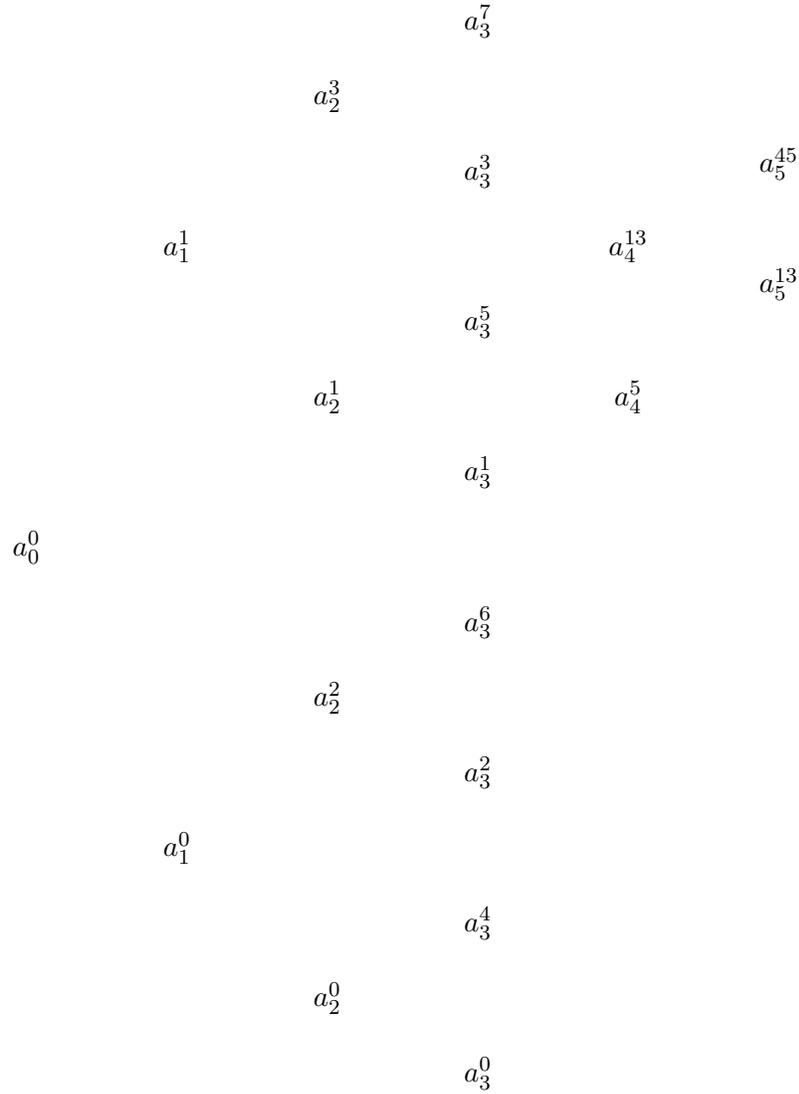
\begin{graph}(13,15)(-4.5,-6.5)
\graphnodesize{.05} \roundnode{-4,0}(-4,0) \roundnode{-2,4}(-2,4)
\roundnode{-2,0}(-2,0) \roundnode{-2,-4}(-2,-4)
\roundnode{0,6}(0,6) \roundnode{0,4}(0,4) \roundnode{0,2}(0,2)
\roundnode{0,-2}(0,-2) \roundnode{0,-4}(0,-4)
\roundnode{0,-6}(0,-6) %%%%%%%%%%%%%%%%%%%%%%%%%%%%%%%%%%%%%%%%%%
\roundnode{2,7}(2,7) \roundnode{2,6}(2,6) \roundnode{2,5}(2,5)
\roundnode{2,3}(2,3) \roundnode{2,2}(2,2) \roundnode{2,1}(2,1)
\roundnode{2,-7}(2,-7) \roundnode{2,-6}(2,-6)
\roundnode{2,-5}(2,-5) \roundnode{2,-3}(2,-3)
\roundnode{2,-2}(2,-2) \roundnode{2,-1}(2,-1) %%%%%%%%%%%%%%%%%%%%%%%%%
\roundnode{4,7}(4,7) \roundnode{4,5}(4,5) \roundnode{4,3}(4,3)
\roundnode{4,1}(4,1) \roundnode{4,-7}(4,-7) \roundnode{4,-5}(4,-5)
\roundnode{4,-3}(4,-3) \roundnode{4,-1}(4,-1) \roundnode{4,4}(4,4)
\roundnode{4,2}(4,2)%%%%%%%
\roundnode{6,4.5}(6,4.5) \roundnode{6,4}(6,4) \roundnode{6,2}(6,2)
\roundnode{6,3.5}(6,3.5) \roundnode{8,4.5}(8,4.5)
\roundnode{8,3.5}(8,3.5) %%%%%%%%%%%%%%%%%%%%%%%%%%%%%%%%%%%%%%%%%%%%%%%%%%%%%%%%%%%%%%
%%%%%%%%%% Dotted Edges %%%%%%%%%%%%%%%%%%%%%%%%%%%%%%%%%%%%
 \edge{-2,0}{-2,-4}[\graphlinedash{1 1}] \edge{-2,-4}{0,-4}[\graphlinedash{1 1}]
\edge{0,6}{0,4}[\graphlinedash{1 1}]
\edge{0,-2}{0,-6}[\graphlinedash{1 1}]
\edge{0,6}{2,6}[\graphlinedash{1 1}]
\edge{0,-2}{2,-2}[\graphlinedash{1 1}]
\edge{0,-6}{2,-6}[\graphlinedash{1 1}]
\edge{2,7}{2,5}[\graphlinedash{1 1}]
\edge{2,2}{2,1}[\graphlinedash{1 1}]
\edge{2,-1}{2,-3}[\graphlinedash{1 1}]
\edge{2,-5}{2,-7}[\graphlinedash{1 1}]
\edge{2,7}{4,7}[\graphlinedash{1 1}]
\edge{2,5}{4,5}[\graphlinedash{1 1}]
\edge{2,1}{4,1}[\graphlinedash{1 1}]
\edge{2,-1}{4,-1}[\graphlinedash{1 1}]
\edge{2,-3}{4,-3}[\graphlinedash{1 1}]
\edge{2,-5}{4,-5}[\graphlinedash{1 1}]
\edge{2,-7}{4,-7}[\graphlinedash{1 1}]
\edge{4,3}{4,2}[\graphlinedash{1 1}]
\edge{4,2}{6,2}[\graphlinedash{1 1}]
\edge{6,4}{6,4.5}[\graphlinedash{1 1}]
\edge{6,4.5}{8,4.5}[\graphlinedash{1 1}] %%%%%%%%%%%%%%%%%%%%%%%%%%%
%%%%%%%%%%%%%%%%%  Thick Edges %%%%%%%%%%%%%%%%%%%%%%%%%%%%%%%%%%%%%%%%%%%%
\edge{-4,0}{-2,0}[\graphlinewidth{.06}]
\edge{-2,0}{-2,4}[\graphlinewidth{.06}]
\edge{-2,4}{0,4}[\graphlinewidth{.06}]
\edge{0,4}{0,2}[\graphlinewidth{.06}]
\edge{0,2}{2,2}[\graphlinewidth{.06}]
\edge{2,2}{2,3}[\graphlinewidth{.06}]
\edge{2,3}{4,3}[\graphlinewidth{.06}]
\edge{4,3}{4,4}[\graphlinewidth{.06}]
\edge{4,4}{6,4}[\graphlinewidth{.06}]
\edge{6,4}{6,3.5}[\graphlinewidth{.06}]
\edge{6,3.5}{8,3.5}[\graphlinewidth{.06}]%%%%%%%%%%%%%%%%%%%%%%%%%%%%%%%
%%%%%%%%%%%%%%%%%%%%%%%%%%%%%%% Labels %%%%%%%%%%%%%%%%%%%%%%%%%%%%%%%%%%%
\bowtext{-4,0}{-2,0}{-0.15}{$a_0^0$}
\bowtext{-2,4}{0,4}{-0.15}{$a_1^1$}
\bowtext{-2,-4}{0,-4}{-0.15}{$a_1^0$}
\bowtext{0,6}{2,6}{-0.15}{$a_2^3$}
\bowtext{0,2}{2,2}{-0.15}{$a_2^1$}
\bowtext{0,-2}{2,-2}{-0.15}{$a_2^2$}
\bowtext{0,-6}{2,-6}{-0.15}{$a_2^0$}
\bowtext{2,7}{4,7}{-0.15}{$a_3^7$}
\bowtext{2,5}{4,5}{-0.15}{$a_3^3$}
\bowtext{2,3}{4,3}{-0.15}{$a_3^5$}
\bowtext{2,1}{4,1}{-0.15}{$a_3^1$}
\bowtext{2,-1}{4,-1}{-0.15}{$a_3^6$}
\bowtext{2,-3}{4,-3}{-0.15}{$a_3^2$}
\bowtext{2,-5}{4,-5}{-0.15}{$a_3^4$}
\bowtext{2,-7}{4,-7}{-0.15}{$a_3^0$}
\bowtext{4,4}{6,4}{-0.15}{$a_4^{13}$}
\bowtext{4,2}{6,2}{-0.15}{$a_4^5$}
\bowtext{6,4.5}{8,4.5}{0.15}{$a_5^{45}$}
\bowtext{6,3.5}{8,3.5}{-0.15}{$a_5^{13}$}

\end{graph}

\vspace{1in}

Now if $\alpha\neq\beta$, then there is some $n$ for which
$\alpha(n)\neq\beta(n)$.  Notice that $f_\alpha(k)\neq f_\beta(k) \
\forall k \geq n$ by construction.  Since $\F$ is nonprincipal,
$X\in \F \Rightarrow |X|=\omega$.  Since $f_\alpha$ and $f_\beta$
agree on at most finitely many integers, the set on which they agree
is not in the ultrafilter, hence $f_\alpha \not\sim_\F f_\beta$.
Thus the map $ \alpha \longmapsto f_\alpha \big/ \F$ is injective.
\qedhere
\end{proof}

\begin{cor} Let $\{A_i\}_{i\in\omega}$ be a set of finite sets
such that the sizes of the $A_i \,$'s is unbounded.  Then a
non-principal ultraproduct of the $A_i \,$s has cardinality
$2^{\aleph_0}$.
\end{cor}

\begin{lemma} Consider $\{\E\}_{n\in I}$, $I\subseteq\omega$, $I$
infinite.  Let $\F$ be a non-principal ultrafilter on $I$. Then
$\ds \prod_{n\in I} \E \bigg/ \F$ embeds in $\EE$, where the set
of atoms of $\EE$ is $\ds \prod_{n\in I} \text{At} \E \bigg/ \F$.
\end{lemma}

\begin{proof} Each $\E$ is atomic.  The property of being atomic
is expressible by the sentence $\varphi = (\forall
x)(\neg(x=0)\Rightarrow (\exists y)((\neg(y=0)\meet
x+y=x)\meet(\forall z)(z+y=y \Rightarrow z=0))$. Since each
$\E\models\varphi$, $\prod_{n\in I} \E \big/ \F \models\varphi$,
since satisfaction of sentences is preserved under the ultraproduct
construction (by Los' lemma--see \cite{CK}).  Hence the ultraproduct
is atomic.

Let $f\in \prod \E$; let $[f]$ be the equivalence class of $f$ in
the ultraproduct.  If $[f]$ is an atom of the ultraproduct, then
$\{n\in I : f(n) \text{ is an atom of } \E \} \in \F$ .  Thus
$\exists g \in \prod \text{At}\, \E \subseteq \prod \E$, where
$[f]=[g]$.  So there is a 1-1 correspondence between atoms of the
ultraproduct and the atoms of $\EE$.  If $[I]$ is the identity in
the ultraproduct, then $I(n)= \1_{A_n}$ ``almost everywhere."  So
$[I]$ corresponds to some $g\in \prod \text{At}\, \E$, namely the
``constant" function $I \longmapsto \1_{A_n}$, and $[I] = [I
\longmapsto \1_{A_n}]$.  So the correspondence preserves the
identity.

Since conversion is identification, it is preserved.

To see that relative multiplication is preserved, notice that it
is defined in terms of the boolean operations.  Now $\EE$ is
\emph{complete}, so the correspondence on atoms extends to an
injective homomorphism from the ultraproduct to $\EE$ that
preserves the boolean operations, hence ;.

Therefore $\EE$ contains an isomorphic copy of the ultraproduct
$\ds \prod_{n\in I} \E \bigg/ \F$. \qedhere

\end{proof}

\begin{lemma} $\EE$ is representable over $\mathbb{R}^2$.
\end{lemma}

\begin{proof}
The algebra $\EE$ has $2^{\aleph_0}$ atoms.  Associate each
diversity atom $a$ with a number in
$\mathbb{R}\cup\{\infty\}$.\footnote{The author may be forgiven,
he hopes, for referring to $\infty$ as a ``number," a clear abuse
of the word.  At least he has not claimed that rational functions
have discontinuities, as in done in so many calculus texts.}  Send
$\1 \in \EE$ to $\{\big\langle\langle x,y \rangle, \langle x,y
\rangle \big\rangle : \langle x,y \rangle \in \mathbb{R}^2 \}$.
For an atom $a\leq \0$, send $a$ to $\ds\big\{\big\langle\langle
x_0,y_0 \rangle, \langle x_1,y_1 \rangle \big\rangle : \langle
x_0,y_0 \rangle\neq \langle x_1,y_1 \rangle,
\frac{y_1-y_0}{x_1-x_0} = a \in \mathbb{R}\cup\{\infty\} \,
\big\}$.  (We take $\dfrac{y_1-y_0}{x_1-x_0}= \infty$ when
$y_1\neq y_0 \meet x_1=x_0$.)  So two points in the plane are
related via $a$ iff they are connected by a line of slope $a$.
Clearly, the (representations of) the atoms are disjoint and their
union is all of $\mathbb{R}^2 \times \mathbb{R}^2$.  It is easy to
see that conversion is identification.  To see that ; works
properly, suppose we have $\alpha, \beta \in \mathbb{R}^2$, and
$\langle \alpha, \beta \rangle \in a;a$.  Then either $\alpha =
\beta$ or $\alpha$ and $\beta$ are connected by a line of slope
$a$.  So $\langle \alpha, \beta \rangle \in a$, and $a;a = \1 +
a$.  If $a,b$ are diversity atoms, $a\neq b$, and $\langle \alpha,
\beta \rangle \in a;b$, then there is some point $\gamma$ in the
plane so that $\alpha$ is connected to $\gamma$ by a line of slope
$a$ and $\gamma$ is connected to $\beta$ by a line of slope $b$.
We cannot have $\alpha = \beta$; also it is clear that $\alpha$
and $\beta$ could be connected by some $\gamma$ given any $c;d$.
So $a;b = \0 \cdot \overline{a+b}$ as desired.

Thus we have represented the atoms.  By the completeness (in the
boolean sense) of $\RR(\mathbb{R}^2)$, we represent any $x\in\EE$
as the join of the (representations of) the atoms below it.
\qedhere
\end{proof}

\begin{thm} \textsf{RRA} is not finitely axiomatizable in
first-order logic.  \end{thm}

\begin{proof}  Let $I= \{n\in\omega : n \geq 5 \text{ and there is
no projective plane of order $n-1$} \}$.  By \cite{BR}, $I$ is
infinite. By 4.1.2, $\{ \E \}_{n\in I}$ is a collection of
 finite non-representable relation algebras of unbounded sizes.
 By lemma 4.2.4, a non-principal ultraproduct $\prod_{n\in I} \E
 \big/ \F$ has an isomorphic copy inside $\EE$.  $\EE$ is
 representable, and then so is $\prod_{n\in I} \E
 \big/ \F$.  Let $\textsf{RRA}^c$ denote the complement of
 \textsf{RRA} relative to the class of all algebras of relational
 type.  If \textsf{RRA} were finitely axiomatizable, then it would
 be axiomatizable by some sentence $\varphi$.  (Just let $\varphi$
 be the conjunction of the finitely many axioms.)  In this case,
 $\textsf{RRA}^c$ would be axiomatized by $\neg\varphi$, hence
 closed under ultraproducts, since that construction preserves
 satisfaction of sentences.  Since $\textsf{RRA}^c$ is \emph{not}
 closed under ultraproducts, there is no such sentence $\varphi$
 and \text{RRA} is not finitely axiomatizable.
\end{proof}

\bigskip
%%%%%%%%%%%%%%%%%%%%%%%%%%%%%%%%%%%%%%%%%%%%%%%%%%%%%%%%%%%%%%%%%%%%%%%%%%%%%%%%%%%%%%%%%%%%%%%%%%%%%%%%%%%%
\section{Equational bases}

In this section we will show that any equational axiomatization of
\textsf{RRA} must be unbounded in the number of variables in the
equations. Consider $\E$ as defined previously.

\begin{lemma} $\A \subsetneq \E \Longrightarrow \A \in
\textsf{RRA}$.  \end{lemma}

\begin{proof} Since $\A$ is finite, $\A$ is atomic.  Let $a_1, a_2,
\ldots ,a_n$ be the diversity atoms of $\E$.  At least one atom of
$\A$ must be the join of at least two atoms of $\E$. (If not, then
$\A = \E$.)  Suppose this atom is $(a_1+a_2)\in \text{At}\, \A$.
This atom is ``big."  Give $a_1+a_2$ the name $a$.  Then $a$
satisfies $$a;a=1  \quad\text{and}\quad a;a_j=\0\cdot\overline{a}_j
\ \ \text{for } j>2$$  Now choose $k>n+1$ so that $\Ek$ is
representable.  Let $b_1, \ldots, b_k$ be the diversity atoms of
$\Ek$.  Define $b=b_1+\ldots+b_{k-n+1}$.  Then we have $b;b=1
\quad\text{and}\quad b;b_j=\0\cdot\overline{b}_j \ \ \text{for }
j>k-n+1$:

\begin{align*}
b;b &= (b_1+\ldots+b_{k-n+1});(b_1+\ldots+b_{k-n+1}) &&\text{by
def'n} \\
    &= \sum_{i,j\leq k-n+1} b_i;b_j &&\text{dist.}\\
    &\geq b_1;b_1 + b_1;b_2 + b_2;b_2 \\
    &= (b_1 + \1) + (\0\cdot\overline{b_1+b_2}) + (b_2 + \1) &&\text{def. of mult. in }\Ek \\
    &= \1 + b_1 + b_2 + \0\cdot\overline{b_1+b_2} \\
    &= 1 \\
\end{align*}
Then also we have for $j>k-n+1$

\begin{align*}
  b;b_j &= \sum_{i\leq k-n+1} b_i;b_j &&\text{dist.}\\
        &= \sum_{i,j\leq k-n+1} \0\cdot \overline{b_i+b_j} &&\text{def. of mult.}\\
        &= \0\cdot \bar{b}_j
\end{align*}

Thus the behavior of $a\in\E$ and $b\in\Ek$ in their respective
algebras is the same.  Thus the mapping
\begin{align*}
 a &\mapsto b \\
 a_j &\mapsto
b_{j+k-n} &&\text{for } j>k-n+1
\end{align*}
on the atoms of $\E$ establishes an isomorphism from $\A$ to the
subalgebra of $\Ek$ with atoms $b, b_{k-n+1}, \ldots, b_k$.  Since
$\Ek$ is representable, so also are its subalgebras, and hence $\A$
is representable also.

\end{proof}

Now if we take $I=\{n\in\omega : n\geq 5 \text{ and there is no
projective plane of order } n-1 \}$, then $\{\E\}_{n\in I}$ is a set
of arbitrarily large finite non-representable relation algebras all
of whose proper subalgebras are representable.  The existence of
such a set gives us the following.

\begin{thm}[J\'{o}nsson '91; Tarski '74] \textsf{RRA} has no
$n$-variable equational basis for $n < \omega$. \end{thm}

A proof of this theorem first appeared in print in \cite{Jon91}, but
was known to Tarski previously--he mentioned it in a taped lecture
in 1974.

\begin{proof}
Let $\Sigma$ be a set of equations with no more than $n$ variables.
Suppose $\textsf{RRA}\models\Sigma$.  Choose $k \in I$ so large that
$2^{2^n} < 2^{k+1}$.  Then $\Ek \notin\textsf{RRA}, \ |\Ek | =
2^{k+1}$.  Take $\varepsilon\in\Sigma$.  Choose $x_1 ,\ldots , x_n
\in \Ek$.  Let $\B=\text{Sg}^{\Ek}(x_1 ,\ldots , x_n)$. Since $\B$
is generated by $n$ elements, we know from boolean-algebraic
considerations that $|B|\leq 2^{2^n} < 2^{k+1} = |\Ek|$, and so $\B
\subsetneq \Ek$. Hence $\B\in \textsf{RRA}$. Thus $\B \models
\varepsilon \, [x_1 ,\ldots , x_n]$.\footnote{$\B \models
\varepsilon \, [x_1 ,\ldots , x_n]$ means that the equation is
satisfied when $x_1$ through $x_n$ are substituted into it.} But
since $\B\subseteq \Ek$, $\Ek \models \varepsilon \, [x_1 ,\ldots ,
x_n]$. $x_1 ,\ldots , x_n$ were arbitrary, so we get $\Ek \models
\varepsilon$, but $\Ek \notin \textsf{RRA}$.  Hence $\Sigma$ does
not axiomatize \textsf{RRA}. \qedhere
\end{proof}

\noindent\textbf{Open Question:}  Is there an $n$-variable
first-order axiomatization of \textsf{RRA}?

\noindent\textbf{Guess:}  This seems unlikely.

For more on this problem see \cite{HH}, chapter 21.

%auto-ignore
% Appendix1 file from standard thesis template
%\appendixtitle
%\appendix
\specialchapt{Appendix: A proof of Birkhoff's Theorem}

\renewcommand{\theenumii}{\arabic{enumii}}
\renewcommand{\labelenumii}{(\theenumii.)}

Birkhoff's Theorem says that varieties are defined by equations and
conversely. (See \cite{Birk}) First we must say what exactly we mean
by ``equations." An equation is a pair of terms: so
$f(x_1,x_2)=g(x_3,h(x_4))$ is associated with $\langle f(x_1,x_2) ,
g(x_3,h(x_4)) \rangle$.  A term will be an element of the absolutely
free algebra of type $\rho$ with generating set $\omega$, denoted by
$\underline{\mathrm{Fr}}_\omega^\rho$. The elements of $\omega$ are
the variables.  An assignment of variables is a function
$f:\omega\to \underline{A}$ for some algebra $\underline{A}$ with
type $\rho$. Any such $f$ extends to a homomorphism
$\hat{f}:\underline{\mathrm{Fr}}_\omega^\rho \to \underline{A}$
which is an assignment of terms to elements of the algebra
$\underline{A}$. To say that an equation $\varepsilon$ is valid in
an algebra is to say that given any assignment of the variables
$f:\omega\to \underline{A}$, $\hat{f}$ sends both sides of the
equation $\varepsilon$ to the same element of the algebra. More
precisely, for $\varepsilon = \langle \varepsilon_0 , \varepsilon_1
\rangle \in \underline{\mathrm{Fr}}_\omega^\rho \times
\underline{\mathrm{Fr}}_\omega^\rho$\,, \  $\underline{A} \models
\varepsilon$ if for all homomorphisms
$h:\underline{\mathrm{Fr}}_\omega^\rho \to \underline{A}$,
$h(\varepsilon_0) = h(\varepsilon_1)$.

\begin{thm}
Let $\Sigma$ be a set of equations ($\Sigma \subset
\mathrm{Fr}_\omega^\rho \times \mathrm{Fr}_\omega^\rho$). Let $\K
= \{ \underline{A} \, : \, \underline{A} \models \Sigma \}$.  Then
$\K= \hh\K = \ss\K = \pp\K$.
\end{thm}

\begin{proof}[Proof] Let $\varepsilon = \langle \varepsilon_0,
\varepsilon_1 \rangle \in \Sigma$. \\

\begin{enumerate}

\item First we show $\K=\ss\K$.  (It is only
necessary to show $\K\supseteq\ss\K$.)\\

Let $\A\in\K$, and $\B\subseteq\A$.  Let $\psi :\Fr \to \B$ be a
homomorphism. Let $\varphi$ be the inclusion
$\B\hookrightarrow\A$.

\[
\begin{diagram}
\node{\Fr} \arrow{s,l}{\psi} \arrow{se,t,..}{\varphi \, \circ \, \psi} \\
\node{\B} \arrow{e,b,J}{\varphi} \node{\A}\\
\end{diagram}
\]

Then $\varphi\circ\psi$ is a homomorphism from $\Fr$ to $\A$.
Since $\A\models\varepsilon$, $\varphi\circ\psi(\varepsilon_0) =
\varphi\circ\psi(\varepsilon_1)$.  But $\varphi$ is the inclusion
map, so $\psi(\varepsilon_0) = \psi(\varepsilon_1)$, and
$\B\models\varepsilon$.  So $\B\in\K$.

%%%%%%%%%%%%%%%%%%%%%%%%%%%%%%%%%%%%%%%%%%%%%%%%%%%%%%%%%%%%%%%%%%%%%%%%%%%%%%%%%%%%%%%%%%%%%%%%%%%%%%%%%%%%%%%%%%%%%%%%%%%

\item $\pp\K=\K$.\\

$\forall j \in J$, let $\A_j \in \K$.  Consider $\prod \A_j = \{
\varphi: J \to \cup A_j \, \big| \, \varphi(j) \in A_j \}$.  Let
$\psi :
\Fr \to \prod \A_j$.  \\

$\forall j\in J \ \exists \psi_j = \pi_j \circ \psi$, where
$\pi_j$ is the projection homomorphism onto $\A_j$.

\[
\begin{diagram}
\node{\Fr} \arrow{s,l}{\psi} \arrow{se,t,..}{\psi_j} \\
\node{\Pi \A_j} \arrow{e,b}{\pi_j} \node{\A_j}\\
\end{diagram}
\]

Now for $\varepsilon = \langle \varepsilon_0, \varepsilon_1
\rangle$, let $\psi(\varepsilon_0)= \varphi_0:J\to \cup A_j$.  By
the diagram, $\varphi_0$ is given by $j
\overset{\varphi_0}{\longmapsto} \psi_j(\varepsilon_0) $.
Similarly, $\psi_1(\varepsilon_1)
 = j \overset{\varphi_1}{\longmapsto} \psi_j(\varepsilon_1)$.\\

 Now we know that for all $j\in J$, $\psi_j(\varepsilon_0) =
 \psi_j(\varepsilon_1)$, since $\varepsilon$ holds in each $\A_j$.
  Thus $\varphi_0$ and $\varphi_1$ agree at each $j\in J$, and so
  $\varphi_0=\varphi_1$, which means that
  $\psi(\varepsilon_0)=\psi(\varepsilon_1)$, and so $\prod \A_j
  \models \varepsilon$.

%%%%%%%%%%%%%%%%%%%%%%%%%%%%%%%%%%%%%%%%%%%%%%%%%%%%%%%%%%%%%%%%%%%%%%%%%%%%%%%%%%%%%%%%%%%%%%%%%%%%%%%%%%%%%%%%

\item $\hh\K=\K$.\\

Let $\A \in \K$, and let $\varphi: \A \xrightarrow{\text{onto}}
\B$, so $\B \in \mathsf{HK}$.  Let $\varepsilon\in\Sigma$,
$\A\models\varepsilon$.  Thus for all hom's $\psi:\Fr \to \A$,
$\psi(\varepsilon_0)=\psi(\varepsilon_1)$,
$\varepsilon=\langle\varepsilon_0, \varepsilon_1 \rangle$.  \\

Let $\psi:\Fr \to \B$.  We show
$\psi(\varepsilon_0)=\psi(\varepsilon_1)$.

\[
\begin{diagram}
\node{\Fr} \arrow{s,l}{\psi}  \\
\node{\B}  \node{\A} \arrow{w,b,A}{\varphi} \\
\end{diagram}
\]

Now we define an assignment of variables $f:\omega\to\A$ by

\[
n \in \omega \overset{f}{\longmapsto} \text{some} \ a \in
\varphi^{-1}\big[ \psi(n) \big]
\]

(So $f$ is a choice function from $\omega$ to $\underset{n
\in\omega}{\bigcup} \varphi^{-1}\big[ \psi(n) \big] \subseteq A$.)
Then $f$ extends to a hom $\hat{f}:\Fr\to\A$, and $\varphi \circ
f(n) = \psi(n)$ for $n \in \omega$, so $\varphi \circ f = \psi
\big|_\omega$.

\[
\begin{diagram}
\node{\Fr}  \arrow{se,t,..}{\hat{f}} \\
\node{\omega} \arrow{n,l,L}{} \arrow{e,b}{f} \node{\A}  \\
\end{diagram}
\]

 Therefore $\varphi \circ \hat{f} = \psi$ since they
agree on the generating set of $\Fr$. Thus the following diagram
commutes:

\[
\begin{diagram}
\node{\Fr} \arrow{s,l}{\psi} \arrow{se,t,..}{\hat{f}} \\
\node{\B}  \node{\A} \arrow{w,b,A}{\varphi} \\
\end{diagram}
\]

Then $\psi(\varepsilon_0) = \varphi \circ \hat{f}(\varepsilon_0) =
\varphi \circ \hat{f}(\varepsilon_1) = \psi(\varepsilon_1)$.  Then
$\varepsilon$ is valid in $\B$ as well, and $\B \in \K$.

\end{enumerate}
\end{proof}

\begin{thm}[Birkhoff, '35]
Let $\mathsf{K}$ be a class of similar algebras (with similarity
type $\rho:I \to \omega$).  Suppose that $\K= \hh\K = \ss\K
=\pp\K$. Then $\K = \{ \underline{A} \, : \, \underline{A} \models
\mathrm{Eq}(\K)\}$, where $\mathrm{Eq}(\K)$ is the set of
equations true in $\K$.

\end{thm}

%%%%%%%%%%%%%%%%%%%%%%%%%%%%%%%%%%%%%%%%%%%%%%%%%%%%%%%%%%%%%%%%%%%%%%%%%%%%

\begin{proof}[Proof]

Let $\Sigma = $Eq$(\K)$.  Let $\A \models \Sigma$.  Show
$\A\in\K$. Thus $\K \supseteq \{\A \, : \, \A\models\Sigma\}$, and
$\K = \{\A \, : \, \A\models\Sigma\}$, since clearly $\K \subseteq
\{\A \, : \, \A\models\Sigma\}$.

\begin{enumerate}

\item Let $\A \models \Sigma$.

\item Construct $\Fra$, the absolutely free algebra generated by
$\A$.

\item Let $I= \{ C \in
\text{Con}(\Fra)\footnote{For any algebra $\A$, Con($\A$) is the set
of congruences on $\A$.} \, : \, \Fra / C \in \K \}$. $I\neq
\emptyset$, since $\Fra \times \Fra$ is a congruence, and so $\Fra
/(\Fra \times \Fra) \iso \underline{1} \in \K$, since $\K$ is closed
under $\hh$.

\item Let $\ds \B = \prod_{C \in I} \Fra / C \in \pp\K=\K$.

\item Let $f:\A\to\B$, $a \overset{f}{\longmapsto} \langle a/C :
C\in I \rangle$.
 Let $\ess = \text{Sg}^{\B} \{ \langle a/C : C\in I \rangle \in \B
\, : \, a \in \A \}$.    So $\ess$ is the subalgebra of $\B$ that
is generated by the range of $f$. We can also write $f:\A \to
\ess$.

\item  $f$ extends to a homomorphism $h:\Fra \to \ess$, $f
\subseteq h$.

\item Note:  $\ess \in \ss\pp\K=\K$.  We want a homomorphism from
$\ess$ onto $\A$, so that $\A \in \hh\ss\pp\K=\K$.

\item Let $g:\Fra \to \A$ be the homomorphism that extends
$\text{Id}_A$.  Then consider

\[
\begin{diagram}
\node{\Fra} \arrow{s,l,A}{g} \arrow{e,t,A}{h} \node{\ess} \arrow{sw,b,..,-}{h^{-1}\big|g} \\
  \node{\A}  \\
\end{diagram}
\]

Note that $h$ is in fact onto $\ess$:  $f$ maps onto the
generating set for $\ess$, so $h$ maps onto the generated
subalgebra.

\item $\hg$ is a function.\\

First we prove a little lemma:  $\forall x \in \Fra$, $h(x) =
\langle x/C : C\in I \rangle$.\\

Proof:  show set of elements with this property is a subalgebra.
$\supseteq
\A$.\\

Base Case: $x\in\A \Rightarrow h(x) = f(x) = \langle x/C : C\in I
\rangle$. \\

Inductive Case:  Assume for $x,y \in \Fra$, $h(x) = \langle x/C :
C\in I \rangle$ and $h(y) = \langle y/C : C\in I \rangle$. Without
loss of generality we consider a \emph{binary} function symbol
$\beta$.

\begin{align*}
 h(\beta(x,y)) &= \beta(h(x),h(y))  &&\text{\footnotesize(h is a hom.)} \\
     &= \beta(\langle x/C : C\in I \rangle, \langle y/C : C\in I
     \rangle) &&\text{\footnotesize(by inductive hyp.)}\\
     &= \langle \beta(x/C, y/C) : C\in I \rangle  &&\text{\footnotesize(by def. of op's in direct prod.)}\\
     &= \langle \beta(x, y)/C : C\in I \rangle &&\text{\footnotesize(by def. of op's in quotient
     alg.)}
\end{align*}

Thus $h(\beta(x,y)) = \langle \beta(x, y)/C : C\in I \rangle$.\\

\bigskip

Now we proceed to prove that $\hg$ is a function.  We want $x=y
\Rightarrow \hg(x) = \hg(y)$, i.e. $h(x)=h(y) \Rightarrow
g(x)=g(y)$.  \\

So suppose $h(x)=h(y)$.  $x$ and $y$ can be regarded as terms in
an equational language by associating $x=t^{\Fra} (a_1, \ldots,
a_k)$, $a_1, \ldots, a_n \in \A$ with $t^{\Fr} (n_1, \ldots,
n_k)$, $n_1, \ldots, n_k \in \omega$.  So we can say that $\langle
x,y \rangle $  is an equation.

By the previous lemma, we have $\langle x/C : C\in I\rangle = h(x)
= h(y) = \langle y/C : C\in I\rangle$, and so $\forall C\in I$,
$x/C = y/C$.  Thus for any congruence $C\in I$ on $\Fra$,
$x\sim_{_C} y$.  So then suppose we have a homomorphism $\varphi:
\Fra \twoheadrightarrow \underline{Z}$ for $\underline{Z} \in\K$.
$\varphi$ induces a congruence $C = \varphi \big| \varphi^{-1}$ on
$\Fra$ such that $\Fra/C \iso \underline{Z} \in \K$. Thus $C\in
I$, and hence $\varphi(x)=\varphi(y)$, since $x\sim_{_C} y$.  Thus
$\K\models \langle x,y \rangle$, and so $\langle x,y \rangle
\in\Sigma$. Since $\langle x,y \rangle \in\Sigma$, $\A \models
\langle x,y \rangle$.  Thus for any homomorphism from $\Fra$ to
$\A$, $x$ and $y$ map to the same element of $\A$.  Now $g:\Fra
\to \A$, so $g(x)=g(y)$. Hence $\hg$ is functional.

\item  $\hg$ is a homomorphism.  \\

Again, we work only with a \emph{binary} function symbol $\beta$.
We show $\hg(\beta(x,y)) = \beta(\hg(x),\hg(y))$, $x,y \in \ess$.

Consider the following diagram:

\[
\begin{diagram}
\node{\Fra} \arrow{s,l,A}{g} \arrow{e,t,A}{h} \node{\ess} \arrow{sw,b,A}{h^{-1}\big|g} \\
  \node{\A}  \\
\end{diagram}
\]

Note that h is surjective.  Then

\[
\exists \, t \in \Fra \quad h(t) = \beta(x,y)
\]

\[
\begin{diagram}
\node{t} \arrow{s,l,T}{g} \arrow{e,t,T}{h} \node{\beta(x,y)}  \\
  \node{\hg(\beta(x,y))}  \\
\end{diagram}
\]
\pagebreak[3]
\[
\exists\, t_x \in \Fra \quad h(t_x) = x \qquad\qquad \exists\, t_y
\in \Fra \quad h(t_y) = y
\]
\[
\begin{diagram}
\node{t_x} \arrow{s,l,T}{g} \arrow{e,t,T}{h} \node{x} \node{t_y} \arrow{s,l,T}{g} \arrow{e,t,T}{h} \node{y} \\
  \node{\hg(x)}  \node[2]{\hg(y)}\\
\end{diagram}
\]

Now note that $h(\beta(t_x,t_y)) = \beta(h(t_x),h(t_y)) =
\beta(x,y)$. \\

So then

\begin{align*}
\hg(\beta(x,y)) &= g(\beta(x,y)) &\text{\footnotesize(holds since $h(\beta(t_x,t_y)) = \beta(x,y)$)}\\
    &= \beta(g(t_x),g(t_y)) &\text{\footnotesize($g$ is a hom.)}\\
    &= \beta(\hg(x),\hg(y)) &\text{\footnotesize($h(t_x)=x$, $h(t_y) = y$)}
\end{align*}

Therefore $\hg$ is a homomorphism.

\end{enumerate}

Thus $\A = \hg[\ess] \in \ss\pp\K$, so $\A \in \hh\ss\pp\K = \K$.
So $\A$ is in $\K$ and $\K$ is defined by equations. \qedhere

\end{proof}

%auto-ignore
% An example bibliography from the standard thesis template
%\renewcommand{\bibname}{\centerline{BIBLIOGRAPHY}}
\unappendixtitle
%\interlinepenalty=300
% For no page break use thebibnopage environment

%\include{acknowl}
\end{document}